\documentclass[a4paper,reqno,10pt]{amsart}
\usepackage {amsmath,amssymb}
\usepackage{amsfonts}
\usepackage{amsthm}
\usepackage{tikz}
\usetikzlibrary{calc}
\usepackage{graphicx,color}
\usepackage [latin1]{inputenc}
\usepackage {times}
\usepackage {float}
\usepackage{hyperref}
\usepackage{amscd}
\usepackage{faktor}
\usepackage{stmaryrd}
\usepackage[nobysame]{amsrefs}

\usetikzlibrary{matrix,arrows}

\numberwithin{equation}{section}
\numberwithin{figure}{section}

\newcommand{\R}{\mathbb R}
\newcommand{\N}{\mathbb N}
\newcommand{\Q}{\mathbb Q}

\newcommand{\Z}{\mathbb Z}

\newcommand{\betrag}[1]{\left|#1\right|}
\DeclareMathOperator{\trop}{B}
\newcommand{\mn}{\mathcal{M}_{n}}
\newcommand{\mnp}{\mathcal{M}_{n+1}}
\DeclareMathOperator{\ft}{ft}
\DeclareMathOperator{\Zy}{Z}
\DeclareMathOperator{\rank}{r}
\newcommand{\mf}{\mathcal{B}}
\newcommand{\F}{\mathcal{F}}

\DeclareMathOperator{\Co}{C}
\DeclareMathOperator{\Sn}{S}
\newcommand{\curlyx}{\mathcal X}
\newcommand{\curlyy}{\mathcal Y}
\newcommand{\curlyb}{\mathcal B}
\newcommand{\curlyt}{\mathcal T}

\newcommand{\id}{\textnormal{id}}
\newcommand{\dist}{\textnormal{dist}}
\newcommand{\conv}{\textnormal{conv}}
\newcommand{\Star}{\textnormal{Star}}
\newcommand{\relint}{\textnormal{rel int}}
\newcommand{\gnrt}[1]{\left\langle #1 \right\rangle} 

\newcommand{\arxiv}[1]{%
  \href{http://arxiv.org/abs/#1}{arxiv:#1}%
}

\newcommand{\newl}{\indent\par}

\newtheorem {theorem}{Theorem}[section]
\newtheorem {proposition}[theorem]{Proposition}
\newtheorem {lemma}[theorem]{Lemma}

\newtheorem {corollary}[theorem]{Corollary}
\theoremstyle {definition}
\newtheorem {definition}[theorem]{Definition}
\newtheorem {notation}[theorem]{Notation}
\newtheorem {example}[theorem]{Example}

\theoremstyle {remark}
\newtheorem {remark}[theorem]{Remark}

\newtheorem*{acknowledgement}{Acknowledgement}

\renewcommand{\Im}{\textnormal{Im}}

\begin{document}
\title{Universal families of rational tropical curves}
\author {Georges Francois}
\address {Georges Francois, Fachbereich Mathematik, Technische Universit\"{a}t
  Kaiserslautern, Postfach 3049, 67653 Kaiserslautern, Germany}
\email {gfrancois@email.lu}
\author {Simon Hampe}
\address {Simon Hampe, Fachbereich Mathematik, Technische Universit\"{a}t
  Kaiserslautern, Postfach 3049, 67653 Kaiserslautern, Germany}
\email {hampe@mathematik.uni-kl.de}
\begin{abstract} We introduce the notion of families of $n$-marked smooth rational tropical curves over smooth tropical varieties and establish a one-to-one correspondence between (equivalence classes of) these families and morphisms
from smooth tropical varieties into the moduli space of $n$-marked abstract rational tropical curves $\mn$. \end{abstract}
\thanks{}
\maketitle

\section{Introduction}

The moduli spaces $\mn$ of $n$-marked abstract rational tropical curves have been well known for several years. An explicit description of the combinatorial structure of $\mn$ and its embedding as a tropical fan can be found in \cite{SS04} and \cite{GKM}. However, so far the moduli spaces $\mn$ have only been parameter spaces, i.e.\ in bijection to the set of tropical curves. In classical geometry or category theory, moduli spaces also carry a universal family which induces all possible families via pull-back along a unique morphism into $\mn$. This paper finds a tropical counterpart by giving a suitable definition of a family of tropical curves and proving that the forgetful map $\ft: \mnp \to \mn$ is then indeed a universal family.

After briefly recalling some known facts in section \ref{section_prelim}, we study the construction of a tropical fibre product in the case where all involved varieties are smooth. For this we define the notion of a locally surjective morphism which might be seen as a tropical analogue of flatness. We conclude that when one of the morphisms is locally surjective, the set-theoretic fibre product can indeed be considered as a tropical fibre product (theorem \ref{fibreproduct}). 

In section \ref{section_families} we define families of rational curves. We prove that the forgetful map of the moduli spaces $\mn$ can be made into such a family by constructing appropriate markings (proposition \ref{markingmn}). Finally, we use the fibre product of the previous section to see that each morphism into $\mn$ induces a family of curves (corollary \ref{families_cor_pullback}).

In section \ref{section_fibre_morphism} we establish the inverse operation, namely we prove that each family of $n$-marked curves also gives rise to a morphism into $\mn$. This leads to our main theorem  \ref{maintheorem} which gives a bijection between equivalence classes of families of $n$-marked curves over a smooth variety $B$ and morphisms $B\rightarrow\mn$.

In the last section we prove that there is a bijective pseudo-morphism, a piecewise linear map respecting the balancing condition, between two equivalent families. In case that the domain of one of the families is a smooth variety, this map is even an isomorphism (theorem \ref{equivalencethm}).

We would like to thank our advisor Andreas Gathmann for many helpful discussions and comments and the anonymous referee for many constructive comments on how to improve and generalise this paper.

\section{Preliminaries and notations}\label{section_prelim}
In this section we quickly review some results on tropical intersection theory and the moduli space $\mn$ of $n$-marked abstract rational tropical curves. 

\subsection{Tropical intersection theory:} A weighted polyhedral complex $\curlyx$ in a vector space $V=\R\otimes\Lambda$ associated to a lattice $\Lambda$ is a pure-dimensional rational polyhedral complex in $V$ all of whose maximal cells $\sigma\in\curlyx$ are equipped with an integer weight $\omega_{\curlyx}(\sigma)$. Each cell $\sigma\in\curlyx$ induces a linear subspace $V_{\sigma}$ of $V$ generated by differences of vectors in $\sigma$ and a sublattice $\Lambda_{\sigma}:=V_{\sigma}\cap \Lambda$ of $\Lambda$. If $\tau<\sigma$ is a codimension $1$ face of $\sigma$, then $u_{\sigma/\tau}$ denotes the (primitive) normal vector of $\sigma$ relative to $\tau$. 
A tropical cycle $X$ in $V$ is the equivalence class modulo refinement of a weighted polyhedral complex $\curlyx$ in $V$ that satisfies the balancing condition for each codimension $1$ cell $\tau\in\curlyx^{(\dim \curlyx -1)}:$
\[ \sum_{\sigma\in\curlyx: \sigma>\tau} \omega_{\curlyx}(\sigma)\cdot u_{\sigma/\tau} =0 \in V/V_{\tau}.\]  A tropical variety is a tropical cycle which has only positive weights. A representative $\curlyx$ of a tropical cycle $X$ is called a polyhedral structure of $X$. If $X$ has a polyhedral structure $\curlyx$ which is a fan, then we call $X$ a fan cycle and $\curlyx$ a fan structure of $X$. The support $|X|$ of a cycle $X$ is the union of all maximal cells of non-zero weight in a polyhedral structure of $X$. A tropical cycle $Y$ is a subcycle of a cycle $X$ if $|Y|\subseteq |X|$. The additive group of all $d$-dimensional subcycles of $X$ is denoted $\Zy_d(X)$, where the sum of two cycles is obtained by taking the union of polyhedral complexes and adding weights for appropriate polyhedral structures. A cycle $X$ is called irreducible if $\Zy_{\dim X}(X)=\Z\cdot X$. 
The star $\Star_X(p)$ of the cycle $X$ around the point $p$ is the tropical cycle whose support consists of vectors $v\in V$ such that $p+\epsilon v$ is in $X$ for small (positive) $\epsilon$ and whose weights are induced by the weights of $X$. If $\curlyx,\curlyx'$ are polyhedral structures of two cycles $X,X'$, then the crossproduct $X\times X'$ is given by the polyhedral structure $\curlyx\times\curlyx'$ with weight function $\omega_{\curlyx\times\curlyx'}(\sigma\times\sigma')=\omega_{\curlyx}(\sigma)\cdot\omega_{\curlyx'}(\sigma')$.
More details can be found in \cite{AR}*{section 2} which covers fan cycles, \cite{AR}*{section 5} which introduces abstract cycles (which are more general than cycles in vector spaces), and \cite{disshannes}*{section 1.1 and 1.2} whose notation we follow in this article.

A morphism $f:X\rightarrow Y$ of tropical cycles is a map from $|X|$ to $|Y|$ which is locally integer affine linear; that means it is locally the sum of an integer linear function and a translation by a real vector. One says that $f$ respects the weights if for suitable polyhedral structures $\curlyx,\curlyy$ and for of all maximal cells $\sigma\in\curlyx$ the weights of $\sigma$ and $f(\sigma)$ are equal. The morphism $f$ is an isomorphism if it respects the weights and has an inverse which is also a morphism. The linear part of the affine linear function that describes a morphism $f:X\rightarrow Y$ around a point $p$ in $X$ gives a morphism $\lambda_{f,p}: \Star_X(p) \to \Star_Y(f(p))$ between the stars.

A rational function on a tropical cycle $X$ is a piecewise integer affine linear function $\varphi:|X|\rightarrow \R$; that means there is a polyhedral structure $\curlyx$ of $X$ such that for all $\sigma\in\curlyx$ the restriction of $\varphi$ to $\sigma$ is the sum of an integer linear form $\varphi_{\sigma}\in\Lambda_{\sigma}^{\vee}$ and a real constant. The intersection product $\varphi\cdot X\in\Zy_{\dim X-1}(X)$ is given by the polyhedral structure $\varphi\cdot\curlyx:=\curlyx\setminus \curlyx^{(\dim X)}$ with the weight function 
\[
\curlyx^{(\dim X-1)}\rightarrow \Z, \ \ \ \tau \mapsto \sum_{\sigma\in\curlyx: \sigma>\tau} \omega_{\curlyx}(\sigma)\cdot \varphi_{\sigma}(v_{\sigma/\tau})-\varphi_{\tau}\left(\sum_{\sigma\in\curlyx: \sigma>\tau} \omega_{\curlyx}(\sigma)\cdot v_{\sigma/\tau}\right),
\]
where the $v_{\sigma/\tau}\in V$ are representatives of the normal vectors $u_{\sigma/\tau}$ (\cite[section 3]{AR},\cite[section 1.2]{disshannes}). Note that the support $|\varphi\cdot X|$ is contained in the domain of non-linearity $|\varphi|$ of $\varphi$. The pull-back of a rational function $\varphi$ on $Y$ along a morphism $f:X\rightarrow Y$ is defined as $f^*\varphi:=\varphi\circ f$ and is a rational function on $X$. If $C$ is a subcycle of $X$, then the projection formula states that $$\varphi\cdot f_* C=f_* f^*\varphi \cdot C,$$
where $f_*: \Zy_d(X)\rightarrow \Zy_d(Y)$ denotes the push-forward of cycles discussed in \cite[construction 2.24]{GKM}, \cite[sections 4 and 7]{AR} and \cite[section 1.3]{disshannes}.

\subsection{Matroid varieties:}Matroid varieties $\trop(M)$ which have been studied in \cites{stusolve, ak, troplin, mapoly} constitute an important class of tropical varieties. They have a canonical fan structure $\mf(M)$ which consists of cones \[\langle\F\rangle:=\left\{\sum_{i=1}^p \lambda_i V_{F_i}:\lambda_1,\ldots,\lambda_{p-1}\geq 0, \lambda_p\in\R\right\}\]
corresponding to chains $\F=(\emptyset\subsetneq F_1\subsetneq\ldots\subsetneq\F_{p-1}\subsetneq F_p=E(M))$ of flats of a (loopfree) matroid $M$ having ground set $E(M):=[n]$. Here $V_F=-\sum_{i\in F} e_i$, where $e_1,\ldots,e_n$ form the standard basis of $\R^n$ and all maximal cones of $\mf(M)$ have trivial weight $1$. The fan structure $\mf(M)$ was introduced in \cite{ak} and is often called the fine subdivision. Note that matroid varieties naturally come with a lineality space containing $\R\cdot (1,\ldots,1)$.

A tropical variety $X$ is smooth if it is locally a matroid variety modulo lineality space $\trop(M)/L$ (cf.\ \cite[section 6]{francoisrau}). This means that for each point $p$ in $X$, the star $\Star_X(p)$ is isomorphic to a matroid variety modulo lineality space. Crossproducts and stars of smooth varieties are again smooth varieties. Recall that $L^n_1$ denotes the curve in $\R^n$ which consists of edges $\R_{\leq 0}\cdot e_i,\ i=0,1,\ldots,n$ (all having trivial weight $1$), where $e_1,\ldots,e_n$ form the standard basis of $\R^n$ and $e_0=-(e_1+\ldots+e_n)$. Then smooth curves are exactly the curves which are locally isomorphic to $L^n_1$ for some $n$.

A main property of smooth varieties which will be crucial in the next section is that they admit an intersection product of cycles having the expected properties (\cite[theorem 6.4]{francoisrau} and \cite[section 3]{shaw}). Furthermore, if $f:X\rightarrow Y$ is a morphism of smooth varieties, then we can pull back any cycle $C\in\Zy_{\dim Y-r}(Y)$ to obtain a cycle $f^*(C)\in\Zy_{\dim X-r}( X)$ \cite[definition 8.1]{francoisrau}: More concretely the pull-back is given by
\[ f^*(C):= \pi_*(\Gamma_f\cdot (X\times C)),\]
where $\pi:X\times Y\rightarrow X$ is the projection to $X$, the graph $\Gamma_f\in\Zy_{\dim X}(X\times Y)$ is the push-forward of $X$ along the morphism $x\mapsto (x,f(x))$ and the intersection product is computed on the smooth variety $X\times Y$. 

In the case where only $Y$ is smooth, we can still pull back each point $p$ in $Y$ along $f$ \cite[remark 4.11]{francois}: The smoothness of $Y$ implies that there is a unique cocycle $\varphi\in\Co^{\dim Y}(Y)$ such that $\varphi\cdot Y= p$; therefore, one can define the pull-back of $p$
as $f^*p:=f^*\varphi \cdot X$. Cocycles are locally given by sums of products of rational functions; we can thus use the above formula for rational functions to compute intersection products of cocycles with tropical cycles \cite[definitions 3.13 and 3.24, proposition 3.28]{francois}. The ability to pull back points along morphisms with smooth target cycles will be an essential ingredient to define families of curves in definition \ref{family}.

\subsection{Moduli spaces:}In \cite[section 3]{GKM} the authors map an $n$-marked rational curve to the vector whose entries are pairwise distances of its leaves and use this to give the moduli space $\mn$ of $n$-marked abstract rational tropical curves the structure of a tropical fan of dimension $n-3$ in $Q_n:=\R^{\binom{n}{2}}/\Im(\phi_n)$, where $\phi_n$ maps $x\in\R^n$ to $(x_i+x_j)_{i<j}$. The edges of $\mn$ are generated by vectors $v_{I|n}:=v_I$ (with $I\subsetneq [n], 1<\betrag{I}<n-1$) corresponding to abstract curves with exactly one bounded edge of length $1$ separating the leaves with labels in $I$ from the leaves with labels in the complement of $I$. Furthermore, the relative interior of each $k$-dimensional cone of $\mn$ corresponds to curves with exactly $k$ bounded edges, whose combinatorial type (i.e.\ the graph without the metric) is the same. The forgetful map $\ft_0:=\ft:\mnp\rightarrow\mn$ forgetting the $0$-th marked end is the morphism of tropical fan cycles induced by the projection $\pi:\R^{\binom{n+1}{2}}\rightarrow\R^{\binom{n}{2}}$ \cite[proposition 3.12]{GKM}. Note that, in order to simplify the notations, we equip $\mnp$ with the markings $0,1,\ldots,n$, when we consider the forgetful map.

It was shown in \cite[section 4]{ak} and \cite[example 7.2]{francoisrau} that $\mn$ is even isomorphic to a matroid variety modulo lineality space (this was already hinted at in \cite{kapranov},\cite{SS04}, see also \cite[theorem 5.5.]{tevelev})
and thus admits an intersection product of cycles: if $\trop(K_{n-1})$ denotes the matroid variety corresponding to the matroid $M(K_{n-1})$ associated to the complete graph $K_{n-1}$ on $n-1$ vertices, then $\mn$ is isomorphic to $\trop(K_{n-1})/L$, with $L=\R\cdot(1,\ldots,1)$. Note that the ground set of $M(K_{n-1})$ is the set of edges of $K_{n-1}$, whereas its flats are exactly the sets of edges of vertex-disjoint unions of complete subgraphs of $K_{n-1}$. Concretely, the isomorphism is the restriction to $\trop(K_{n-1})/L$ of the isomorphism
\begin{eqnarray*}
f:  \R^{\binom{n-1}{2}}/ L &\rightarrow & \R^{\binom{n}{2}}/\Im(\phi_n) \\  (a_{i,j})_{i<j} &\mapsto & (b_{i,j})_{i<j} \ , \ \ \ \text{ with } b_{i,j}=\begin{cases} 0, & \text{ if } n\in\{i,j\} \\ 2\cdot a_{i,j}, & \text{ else}\end{cases}.
\end{eqnarray*} 
In this setting the forgetful map is thus induced by the projection $\pi:\R^{\binom{n}{2}}\rightarrow\R^{\binom{n-1}{2}}$.

\section{Tropical fibre products}\label{section_fibre_product}

The aim of this section is to construct a tropical fibre product in the case that all involved cycles are smooth and one of the morphisms is locally surjective: 

\begin{definition}
A morphism $f:X\rightarrow Y$ of tropical varieties is called \emph{locally surjective} if for every point $p$ in $X$, the induced linear map
$$\lambda_{f,p}: \Star_X(p) \to \Star_Y(f(p))$$
is surjective.
\end{definition}

\begin{lemma}\label{flatlemma}
 Let $f: X \to Y$ be a locally surjective morphism. Then the following holds:
\begin{itemize}
 \item Let $\mathcal{X},\mathcal{Y}$ be polyhedral structures of $X$ and $Y$ such that $f(\tau) \in \mathcal{Y}$ for all $\tau \in \mathcal{X}$ (cf.\ \cite[lemma 1.3.4]{disshannes}). For $\tau \in \mathcal{X}$ we have
$$f(U(\tau)) = U(f(\tau)), \text{ where } U(\tau):=\bigcup_{\sigma\in\mathcal{X}: \sigma>\tau} \relint(\sigma).$$ In particular, $f$ is an open map, i.e.\ maps open sets to open sets.
 \item  Let $\varphi$ be a rational function on $Y$. Then the domain of non-linearity of $\varphi \circ f$ is equal to the preimage of the domain of non-linearity of $\varphi$, i.e.\  $$\betrag{\varphi \circ f} = f^{-1}(\betrag{\varphi}).$$
\end{itemize}
\begin{proof}
The first part obviously follows from the local surjectivity of $f$. Note that the set of all possible $U(\tau)$ for all possible polyhedral structures of $X$ forms a topological basis of the standard euclidean topology on $\betrag{X}$. For the second part it suffices to prove that $\varphi$ is locally linear at $p \in Y$ if and only if $\varphi \circ f$ is locally linear at some point $q \in f^{-1}(p)$. But this is already clear from the first part.
\end{proof}
\end{lemma}

\begin{lemma}
\label{settheoreticfibre}
Let $Y$ be a smooth variety and let $f:X\rightarrow Y$ be a locally surjective morphism. Then the intersection-theoretic fibre over each point $y$ in $Y$ has only positive weights and its support agrees with the set-theoretic fibre, that means  $$|f^*(y)|=f^{-1}\{y\}.$$ 
\end{lemma}

In order to prove this we need the following lemma:

\begin{lemma}
\label{max}
Let $M$ be a matroid of rank $r$ on the set $[m]$. Let $L:=\R\cdot (1,\ldots,1)$. Then 
$\max\{x_1,\ldots,x_m\}^{r-1}\cdot\trop(M)=L$. 
\end{lemma}
\begin{proof}
We set $\varphi:=\max\{x_1,\ldots,x_m\}$ and denote by $T(M)$ the truncation of $M$, i.e.\ the matroid obtained from $M$ by removing all flats of rank $r-1$.
Let $\F:=(\emptyset=F_0\subsetneq F_1\ldots\subsetneq F_{r-2}\subsetneq F_{r-1}:=E(M))$ be a chain of flats with $\rank_M(F_i)=i$ for $i\leq j$ and $\rank_M(F_i)=i+1$ for $i\geq j+1$. Note that $\varphi$ is linear on the cones of $\mf(M)$ and satisfies $\varphi(V_F)=-1$ if $F=E(M)$, and $0$ otherwise. As
\[\sum_{F \text{ flat of $M$ with } F_j\subsetneq F \subsetneq F_{j+1}} V_F = V_{F_{j+1}}+(\betrag{F \text{ flat with } F_j\subsetneq F \subsetneq F_{j+1}}-1)\cdot V_{F_j},\]
it follows directly from the definition of intersecting with rational functions that $\varphi\cdot\mf(M)=\mf(T(M))$. Now a simple induction proves the claim.
\end{proof}

\begin{proof}[Proof of lemma \ref{settheoreticfibre}]
Let $y$ be a point in $Y$ and let $x$ be a point in $X$ with $f(x)=y$. As the intersection-theoretic computations are local, it suffices to show the claim for the induced morphism $\lambda_{f,x}$ on the respective stars; that means we can assume that $f$ is linear, $X$ is a fan cycle, $Y$ is a matroid variety modulo lineality space and $y=0$. Let $r$ be the dimension of $Y$.
We choose convex rational functions $\varphi_i$ such that $y=\varphi_1\cdots\varphi_{r}\cdot Y$. This can be done by decomposing $Y$ into a cross product of matroid varieties modulo $1$-dimensional lineality spaces (cf.\ \cite[section 2]{francoisrau}) and then using lemma \ref{max}. We show by induction that $f^*\varphi_i\cdots f^*\varphi_{r}\cdot X$ is a cycle having only positive weights and satisfying
\[|f^*\varphi_i\cdots f^*\varphi_{r}\cdot X|=f^{-1} (| \varphi_i\cdots\varphi_{r}\cdot Y|),\]
which implies the claim because $f^*(y)=f^*\varphi_1\cdots f^*\varphi_{r}\cdot X$: Since $f^*\varphi_{i-1}$ is convex and $f^*\varphi_i\cdots f^*\varphi_{r}\cdot X$ has only positive weights, it follows from \cite[lemma 1.2.25]{disshannes} that the positivity of the weights is preserved and that
\[|f^*\varphi_{i-1}\cdot f^*\varphi_i\cdots f^*\varphi_{r}\cdot X|=|(f^*\varphi_{i-1})_{\mid |f^*\varphi_i\cdots f^*\varphi_{r}\cdot X|}|,\]
where the right hand side is the domain of non-linearity of the restriction of the rational function $f^*\varphi_{i-1}$ to (the support of) $f^*\varphi_i\cdots f^*\varphi_{r}\cdot X$. By induction hypothesis, this is equal to the domain of non-linearity
\[|(\varphi_{i-1}\circ f)_{\mid f^{-1}(|\varphi_i\cdots \varphi_{r}\cdot Y|)}|,\]
which by lemma \ref{flatlemma} coincides with
\[f^{-1}(|{\varphi_{i-1}}_{\mid |\varphi_i\cdots \varphi_{r}\cdot Y|}|)=f^{-1}(|\varphi_{i-1}\cdot\varphi_i\cdots \varphi_{r}\cdot Y|).\]
Note that our induction hypothesis (for stars around different points) and the locality of intersecting with rational functions (cf.\ \cite[proposition 1.2.12]{disshannes}) ensure that the restriction of $f$ to $f^*\varphi_i\cdots f^*\varphi_{r}\cdot X$ satisfies the assumptions of lemma \ref{flatlemma}.
\end{proof}

\begin{remark}
Lemma \ref{settheoreticfibre} ensures that all set-theoretic fibres of a locally surjective morphism have the expected dimension. Therefore, local surjectivity might be seen as a tropical analogue of flatness.
\end{remark}

\begin{definition}
\label{abc}
Let $f:X\rightarrow Y$ and $f':X'\rightarrow Y$ be morphisms of smooth varieties. Assume that $f'$ is locally surjective. Recall that the diagonal $\Delta_Y\in\Zy_{\dim Y}(Y\times Y)$ is just the push-forward of $Y$ along the morphism $y\mapsto (y,y)$. Then we define the tropical \emph{fibre product}
\[ X\times_Y X':= (f\times f')^* (\Delta_Y) \in \Zy_{\dim X+\dim X'-\dim Y}(X\times X') \]
to be the pull-back of the diagonal $\Delta_Y$ along the morphism of smooth varieties $f\times f':X\times X'\rightarrow Y\times Y$. Let $\pi_X, \pi_{X'}$ be the projections from $X\times X'$ to $X$ and $X'$ respectively. As the support of the pull-back satisfies
\[ |(f\times f')^* (\Delta_Y)| \subseteq (f\times f')^{-1} (|\Delta_Y|)=\{(x,x')\in X\times X': f(x)=f'(x')\},\]
we obtain the following commutative diagram of tropical morphisms:
\[
\begin{CD}
X\times_Y X' @>\pi_X>> X\\
@VV\pi_{X'}V @VVfV\\
X' @>f'>> Y
\end{CD}
\] 
\end{definition}

\begin{remark}
We will see later in theorem \ref{fibreproduct} that the assumption that $f'$ is locally surjective is needed to make sure that $X\times_Y X'$ is indeed a fibre product. Therefore, we can only define it for this case. 
\end{remark}

\begin{proposition}
\label{fibreXf}
Using the notations and assumptions of definition \ref{abc} we have $$\pi_X^* (p)=\{p\}\times f'^*(f(p)),$$ for each point $p$ in $X$.
\end{proposition}
\begin{proof}
In this proof, by abuse of notation, $\pi_X,\pi_{X'},\pi_{X\times X'}$ denote projections from a product of $X,Y,X'$ to the respective cycle. Let $\varphi\in\Co^{\dim X}(X)$ be the (uniquely defined) cocycle such that $\varphi\cdot X=p$ \cite[corollary 4.9]{francois}. By the projection formula and commutativity of intersection products \cite[proposition 3.28]{francois} we have
\[\pi_X^* (p) = \pi_X^*\varphi\cdot (X\times_Y X') = (\pi_{X\times X'})_{\ast} \Gamma_{f\times f'}\cdot (\{p\}\times X'\times\Delta_{Y}).\]
Since we know by \cite[theorem 6.4(9) and lemma 8.4(1)]{francoisrau} that
\[\{p\}\times X'\times\Delta_{Y}=(\{p\}\times X'\times Y\times Y)\cdot(X\times X'\times\Delta_{Y})\] and $\Gamma_f\cdot (\{p\}\times Y)=\{(p,f(p)\}$, the above is equal to
\[\{p\}\times (\pi_{X'})_{\ast} ((\Gamma_{f'}\times\{f(p)\})\cdot(X'\times\Delta_{Y})).\]
Now it follows in an analogous way from \cite[theorem 6.4(9) and lemma 8.4(2)]{francoisrau} that the latter equals
\begin{align*}
&\{p\}\times (\pi_{X'})_{\ast} (\Gamma_{(f',f')}\cdot (X'\times Y\times \{f(p)\})) \\
=\;& \{p\}\times (\pi_{X'})_{\ast} (\Gamma_{f'}\cdot (X'\times \{f(p)\})) \\
=\;& \{p\} \times {f'}^*(f(p)).
\end{align*}
\end{proof}

We are now ready to state the main theorem of this section.
\begin{theorem}
\label{fibreproduct}
If $f:X\rightarrow Y$, $f':X'\rightarrow Y$ are morphisms of smooth tropical varieties and $f'$ is locally surjective, then the support of $X\times_Y X'$ is
\[|X\times_Y X'|=\{(x,x')\in X \times X': f(x)=f'(x')\}. \] 
In particular, $X\times_Y X'$ satisfies the universal property of fibre products.
\end{theorem}
\begin{proof}
Combining lemma \ref{settheoreticfibre} and proposition \ref{fibreXf} we immediately obtain that the support of $X\times_Y X'$ is $\{(x,x')\in X \times X': f(x)=f'(x')\}$. For the second part, let
$Z$ be the domain of two tropical morphisms $g:Z\rightarrow X$, $g':Z\rightarrow X'$ such that $f\circ g=f'\circ g'$. Then it is clear that $z\mapsto G(z):=(g(z),g'(z))$ is the only morphism from $Z$ to $X\times_Y X'$ such that $\pi_X\circ G=g$ and $\pi_{X'}\circ G=g'$.
\end{proof}

\begin{remark}
Unfortunately, the tropical fibre product is not uniquely defined by the ``tropical universal property'': Changing the weights of $X\times_Y X'$ in such a way that it still satisfies the balancing condition produces a non-isomorphic cycle that still fulfils the ``tropical universal property''. This happens because a tropical morphism whose inverse is again a morphism is not necessarily an isomorphism. Therefore, one might try to give a slightly stronger definition of a tropical morphism, somehow respecting the weights, in order to fix this flaw. However, since this is far beyond the scope of this paper and we do not actually need the universal property, we do not look further into this.
\end{remark}

\begin{remark}
Let $\curlyx$ and $\curlyx'$ be polyhedral structures of $X$ and $X'$. For two cells $\sigma\in\curlyx$ and $\sigma'\in\curlyx'$ we define the cell 
$\sigma\times_Y \sigma':=\{(x,x')\in \sigma\times \sigma': f(x)=f'(x')\}$. 
By theorem \ref{fibreproduct} $$\curlyx\times_Y\curlyx':=\{\sigma\times_Y\sigma':\sigma\in\curlyx,\sigma'\in\curlyx'\}$$ is a polyhedral structure of $X\times_Y X'$.
\end{remark}

We prove in the next propositions that fibre products are tropical varieties (i.e.\ all weights are positive) and the projections $\pi_X:X\times_Y X'\rightarrow X$ are locally surjective.

\begin{proposition}
\label{weight1}
All maximal cells of $X\times_Y X'$ have positive weight. In particular, $X\times_Y X'$ is a tropical variety.
\end{proposition}
\begin{proof}
Let $\sigma$ be a maximal cell of $\curlyx\times_Y \curlyx'$, where $\curlyx,\curlyx'$ are polyhedral structures of $X,X'$. Let $p$ be a point in the interior of $\sigma$. We know by
Proposition \ref{fibreXf} and Lemma \ref{settheoreticfibre} that the pull-back $\pi_X^*(\pi_X(p))$ of the point $\pi_X(p)$ along the morphism $\pi_X:X\times_Y X'\rightarrow X$ has only positive weights. Set $n:= \dim X+\dim X'-\dim Y$. The locality of the pull-back operation implies that the pull-back of the origin along the morphism $\lambda_{\pi_X,p}: (\omega_{\curlyx\times_Y \curlyx'}(\sigma)\cdot\R^n) \rightarrow \Star_X(\pi_X(p))$ has only positive weights. As there are convex rational functions $\varphi_1,\ldots,\varphi_{\dim X}$ on the smooth cycle $\Star_X(\pi_X(p))$ that cut out the origin and $$(\lambda_{\pi_X,p})^*(0) = \omega_{\curlyx\times_Y \curlyx'}(\sigma) \cdot (\lambda_{\pi_X,p})^*\varphi_1 \cdots (\lambda_{\pi_X,p})^*\varphi_{\dim X} \cdot \R^n,$$ it follows from \cite[Lemma 1.2.25]{disshannes} that the weight $\omega_{\curlyx\times_Y \curlyx'}(\sigma)$ is positive.
\end{proof}

\begin{proposition}
\label{thirdaxiompi}
The projection morphism $\pi_X:X\times_Y X'\rightarrow X$ is locally surjective.
\begin{proof}
Let $p$ be a point contained in some cell $\sigma \times_Y \sigma'$ and let $q \in \alpha$ for some $\alpha \geq \sigma$. Consider $f(q)$ as an element of $\Star_{Y}(f(p))$. By the local surjectivity of $f'$, it has a preimage $v$ under $f'$ in some $\alpha' \geq \sigma'$; so the point $(q,v)$ is in $\Star_{X\times_Y X'}(p)$ and is obviously mapped to $q$ by $\pi_X$.
\end{proof}
\end{proposition}

\section{Families of curves and the forgetful map}\label{section_families}
The aim of this section is to prove that every morphism from a smooth variety $X$ to $\mn$ gives rise to a family of curves. We start by defining families of curves over smooth varieties. 

\begin{definition}[Family of curves]
\label{family}
Let $n\geq 3$ and let $B$ be a smooth tropical variety. A locally surjective morphism $T \stackrel{g}{\to} B$ of tropical varieties is a \emph{prefamily} of $n$-marked tropical curves if it satisfies the following conditions:
\begin{enumerate}
\item For each point $b$ in $B$ the cycle $g^*(b)$ is a smooth rational tropical curve with exactly $n$ unbounded edges that are called the leaves of $g^*(b)$.
\item The linear part of $g$ at  any cell $\tau$ in (some and thus any polyhedral structure of) $T$ induces a surjective map $\lambda_{g\mid\tau}:\Lambda_{\tau}\rightarrow\Lambda_{g(\tau)}$ on the corresponding lattices.
\end{enumerate}  

A \emph{tropical marking} on a prefamily $T \stackrel{g}{\to} B$ is an open cover $\{U_\theta, \theta \in \Theta\}$ of $B$ together with a set of integer affine linear maps $s_i^\theta: U_\theta \to T, i=1,\dots,n$, such that the following holds:
\begin{enumerate}
 \item For all $\theta \in \Theta, i=1,\dots,n$, we have $g \circ s_i^\theta = \id_{U_\theta}$.
 \item For any $b \in U_\theta$ if $l_1,\dots,l_n$ denote the leaves of the fibre $g^*(b)$, then for each $i \in [n]$ there exists exactly one $j \in [n]$ such that $s_j^\theta(b) \in l_i^\circ$, where $l_i^\circ$ denotes the leaf without its vertex.
 \item For any $\theta \neq \zeta \in \Theta$ and $b \in U_\theta \cap U_\zeta$, the points $s_i^\theta(b)$ and $s_i^\zeta(b)$ mark the same leaf of $g^*(b)$. Note that we do not require them to coincide.
\end{enumerate}

A \emph{family} of $n$-marked tropical curves is then a prefamily with a marking.

We call two families $T \stackrel{g}{\to} B, T' \stackrel{g'}{\to} B$ \emph{equivalent} if for any $b$ in $B$ the fibres $g^*(b), g'^*(b)$ are isomorphic as $n$-marked tropical curves.
\end{definition}

\begin{example} \newl
\begin{itemize}
\item The morphism \[\pi: L^n_1\times\R\rightarrow\R, \ (x_1,\ldots,x_n,y)\mapsto y,\] together with the trivial marking $y\mapsto (e_i,y), \ i=0,1,\ldots, n$, is a family of $(n+1)$-marked curves.
\item We consider the tropical curves $X_1:=L^2_1$ and $X_2:=(\R\times \{0\}) + (\{0\}\times\R)$, where the latter is a sum of tropical cycles. Let us consider the morphisms \[\pi_i: L^n_1\times X_i\rightarrow\R,\  (x_1,\ldots,x_n,y_1,y_2)\mapsto y_2.\]
Although $\pi_i^*(p)=L^n_1\times\{p\}$ for all points $p$ in $\R$, $\pi_i$ is not a family of curves: e.g.\ for $i\in \{1,2\}$ and $p = ( (0,\dots,0),(-1,0)) \in L^n_1 \times X_i$ the map
$$\lambda_{\pi_i,p}: \Star_{L^n_1 \times X_i}(p) \cong L^n_1 \times \R \to \Star_{\R}(0) \cong \R$$
is just the constant zero map. Geometrically, we see that the set-theoretic fibre $\pi_i^{-1}\{0\}$ is $2$-dimensional. 
This illustrates the necessity of the local surjectivity without which $\pi,\pi_1,\pi_2$ would be equivalent families with completely different domains $L^n_1\times\R,L^n_1\times X_1,L^n_1\times X_2$ (compare to section \ref{section_equivalence}).
\end{itemize}
\end{example}

\begin{remark}\label{families_remark_second}
While the first condition in the definition of a prefamily is self-explanatory, the second requires some justification. We will see later that for all cells $\tau$ in (a polyhedral structure of) $T$ on which $g$ is not injective, condition (2) is already implied (cf.\ lemma \ref{lambdasurjective}). 
However, we will need condition (2) on all cells $\tau$, including those on which $g$ is injective, to show that the locally affine linear map $B\rightarrow \mn$ induced by the family $T\rightarrow B$ is an integer map and thus a tropical morphism (cf.\ definition \ref{deffibremorph}, proposition \ref{dgaffinelinear}). It is, in fact, not clear to us whether there exists an example of a locally surjective morphism with smooth curves as fibres, where this condition is not fulfilled or whether this condition can actually be dropped.
\end{remark}

We now want to show that the forgetful map $\ft:\mnp\rightarrow \mn$ is a family of $n$-marked curves. Therefore, we prove that it is locally surjective:

\begin{lemma}
\label{thirdaxiomft0}
For $n \geq 3$ and $v \in \mnp$, the map $\lambda_{\ft,v}$ is surjective. Hence the forgetful map is locally surjective.
\begin{proof}
 Let $\tau$ be the minimal cell of $\mnp$ containing $v$ and let $C$ be the curve corresponding to the point $v$. Let $w'$ be an element of $\Star_{\mn}(\ft(v))$. Then $w'$ corresponds to a curve which is obtained from the curve corresponding to $\ft(v)$ by resolving some higher-valent vertices. If we resolve the same vertices in $C$, we get a curve $C'$ corresponding to a point $v' \in \mnp$ such that $\ft(v') = w'$. In particular, the combinatorial type of $C'$ corresponds to a cell $\tau' \geq \tau$, so $v' \in \Star_{\mnp}(v)$.
\end{proof}
\end{lemma}

We compute the fibres of the forgetful map in the following proposition.

\begin{proposition}
\label{forgetful}
Let $\ft:\mnp\rightarrow\mn$ be the forgetful map. Then for each point $p$ in $\mn$, the fibre $\ft^*(p)$ is a smooth rational curve having $n$ unbounded edges. 
\end{proposition}

Our proof makes use of the following lemma.
\begin{lemma}
The edge $\R_{\geq 0}\cdot v_{\{0,n\}}$ has trivial weight $1$ in the fibre $\ft^*(0)$.
\end{lemma}
\begin{proof}
Using the isomorphism $f:\trop(K_n)/L\rightarrow \mnp$ introduced in section 2 we have to compute the fibre over the origin of the projection $\pi:\trop(K_n)/L\rightarrow \trop(K_{n-1})/L$ which forgets the coordinates $x_{0,i}$. Note that we gave $K_n$ and $K_{n-1}$ the respective vertex sets $\{0,1,\ldots,n-1\}$ and $\{1,\ldots,n-1\}$ and that by abuse of notation we denoted both lineality spaces by $L$. If $\tilde{\pi}:\trop(K_n)\rightarrow\trop(K_{n-1})$ is the ``naturally lifted'' projection, then \cite[proposition 8.5]{francoisrau} states that $\pi^*(0)=(\tilde{\pi}^*(L))/L$. This enables us to use lemma \ref{max} to conclude that $\tilde{\pi}^*(L)=\varphi^{n-3}\cdot\trop(K_n)$, where $\varphi:=\max\{x_{i,j}:0<i<j\leq n-1\}$. Let $G$ be the flat of $M(K_n)$ corresponding to the complete subgraph with vertex set $\{1,\ldots,n-1\}$. It is easy to see that $\varphi$ is linear on the cones of $\mf(K_n)$ and that $\varphi(V_F)=-1$ if $F\in\{G,E(K_n)\}$, and $\varphi(V_F)=0$ otherwise.
A straightforward induction shows that the cone associated to $\F:=(\emptyset\subsetneq F_1\subsetneq\ldots\subsetneq F_{n-3-k}\subsetneq G\subsetneq E(K_n))$, where $\rank(F_i)=i$, has weight $1$ in $\varphi^k\cdot\mf(K_n)$. It follows that $\R_{\geq 0}\cdot v_{\{0,n\}}=f(\langle \emptyset \subsetneq G \subsetneq E(K_n)\rangle)$ has weight $1$ in $\ft^*(0)$.
\end{proof}

\begin{proof}[Proof of proposition \ref{forgetful}]
We know from \cite[proposition 2.1.21]{disshannes} that for each $p$ in $\mn$ there is a smooth rational irreducible curve $C_p$ which has $n$ unbounded ends and whose support $\betrag{C_p}$ is equal to the set-theoretic fibre $\ft^{-1}\{p\}$. The edges of $C_0$ are simply $\R_{\geq 0}\cdot v_{\{0,i\}}$, with $i\in [n]$. The local surjectivity of the forgetful map implies that $$\betrag{\ft^*(p)}=\ft^{-1}\{p\}=|C_p|.$$ Therefore, the irreducibility of $C_p$ allows us to conclude that $\ft^*(p)=\lambda_p\cdot C_p$ for some integer $\lambda_p$. Since any two points in $\mn$ are rationally equivalent \cite[theorem 9.5]{francoisrau} and the forgetful map is compatible with rational equivalence \cite[remark 9.2]{francoisrau}, we conclude that $\ft^*(p)$ and $\ft^*(0)$ are rationally equivalent and thus $\lambda_p=\lambda_0$. This finishes the proof as $\lambda_0=1$ by the previous lemma. 
\end{proof}

As the forgetful map clearly fulfils the second axiom on a prefamily, the following corollary is a direct consequence of proposition \ref{forgetful} and lemma \ref{thirdaxiomft0}.

\begin{corollary}
The forgetful map $\ft:\mnp\rightarrow \mn$ is a prefamily of $n$-marked tropical curves.
\end{corollary}

We now want to define a marking on the forgetful map. To do that, we need a basis of the ambient space $Q_n$ of $\mn$. In \cite[section 2]{psiclasses} the authors construct a generating set in the way that we will shortly describe and it is easy to see (e.g.\ by induction on $n$, using the forgetful map) that it becomes a basis if we remove an arbitrary element.

For any $k \in \{1,\dots,n\}$, we set $$V_{k,n}:=V_k := \{v_I; k \notin I, \betrag{I} = 2\}.$$
For any $I_0 \subseteq \{1,\dots,n\}$ with $v_{I_0} \in V_k$ we define
$$V_{k,n}^{I_0}:=V_k^{I_0} := V_k \setminus \{v_{I_0}\}.$$

\begin{lemma}\label{families_lemma_basis}
 Let $v_I \in \mn, I \subseteq [n]$ and assume that $k \notin I$. Then we have
$$v_I = \begin{cases}
         &\sum_{J \subseteq I, v_J \in V_k^{I_0}} v_J, \textnormal{ if } I_0 \nsubseteq I \\
	 - &\sum_{J \nsubseteq I, v_J \in V_k^{I_0}} v_J, \textnormal{ otherwise}
        \end{cases}.$$
\begin{proof}
It was shown in \cite[lemma 2.4, lemma 2.7]{psiclasses} that $\sum_{w \in V_k} w = 0$ and that $v_I = \sum_{v_S \in V_k, S \subseteq I} v_S$. This implies the above equation.
\end{proof}
\end{lemma}

For the following proposition, for each $i = 1,\dots,n$ we fix an arbitrary $I_0(i)$ with $v_{I_0(i)} \in V_{i,n}$ and write $W_{i,n} := V_{i,n}^{I_0(i)}$ for simplicity.

\begin{proposition}
\label{markingmn}
 There exists a tropical marking $s_i^\theta$ on the forgetful map such that, as a marked curve, the fibre over each point $p$ in $\mn$ is exactly the curve represented by that point. In particular, $(\mnp \stackrel{\ft}{\to} \mn, s_i^\theta)$ is a family of $n$-marked rational tropical curves.
\begin{proof}
Again, \cite[proposition 2.1.21]{disshannes} tells us that the fibre over each point is exactly the curve represented by that point (without markings).

The idea of the construction is the following: We define the marking on the basis curves $v_I$ by placing the mark on the $i$-th leaf with a fixed distance $\alpha$ from the vertex of the leaf. However, this cannot work globally: Linearity of the map implies that for some element $v_J$ not in the basis, the mark now actually moves \emph{towards} the vertex when moving outwards along the ray $\gnrt{v_J}$. Since the mark has to stay on the relative interior of the leaf, this means that the map is only feasible on the open subset of points that have distance less than $\alpha$ from the origin. We then cover $\mn$ by these subsets for appropriate $\alpha$ and obtain a marking.

For $\alpha\in \N_{>0}$ we define 
$$U_\alpha := \left\{\sum_{v_I \in \mn} \lambda_I v_I; \lambda_I \geq 0; \sum \lambda_I < \alpha\right\} \cap \betrag{\mn}.$$
Clearly $\{U_\alpha, \alpha \in \N_{>0}\}$ is a cover of $\mn$. Now pick any $\alpha \in \N_{>0}, i \in \{1,\dots,n\}$. We define 
$$s_i^\alpha: U_\alpha \to \mnp, v \mapsto \alpha \cdot v_{\{0,i\}} + A_i(v),$$
where $A_i: Q_n \to Q_{n+1}$ is the linear map defined by $A_i(v_I) = v_{I \mid n+1}$ for all $v_I \in W_{i,n}$. Note that in this proof the $v_I$ represent curves with markings in $\{1,\ldots,n\}$ and thus live in $Q_n$, whereas the $v_{I \mid n+1}$ correspond to curves with markings in $\{0,1,\ldots,n\}$ and thus live in $Q_{n+1}$. We have to show that this defines indeed a map into $\mnp$ and that it is a tropical marking. 

For this, we choose any $v_I \in \mn$ and assume without restriction that $i \notin I$, since $v_I = v_{I^c}$. By lemma \ref{families_lemma_basis} we have
$$v_I = \begin{cases}
         &\sum_{J \subseteq I, v_J \in W_{i,n}} v_J, \textnormal{ if } I_0 \nsubseteq I \\
	 - &\sum_{J \nsubseteq I, v_J \in W_{i,n}} v_J, \textnormal{ otherwise}
        \end{cases},$$
and similarly in $\mnp$:
\begin{align*}
v_{I \mid n+1} &= \left\{
		  \begin{aligned}
                    &\sum_{J \subseteq I, v_J \in W_{i,n+1}} v_J = \sum_{J \subseteq I, v_J \in W_{i,n}} v_{J\mid n+1}, &\textnormal{ if } I_0 \nsubseteq I \\
		   -&\sum_{J \nsubseteq I, v_J \in W_{i,n+1}} v_J = - \sum_{J \nsubseteq I, v_J \in W_{i,n}} v_{J\mid n+1} - \sum_{j \neq 0,i} v_{\{0,j\}}, &\textnormal{ otherwise}
		    \end{aligned}
                   \right. \\
		&= \left\{
		   \begin{aligned}
		    A_i(v_I), &\textnormal{ if } I_0 \nsubseteq I \\
		    A_i(v_I) + v_{\{0,i\}}, &\textnormal{ otherwise (since } \sum_{j=1}^n v_{\{0,j\}} = 0)
		    \end{aligned} \ \ .
		   \right.
\end{align*}
Summarising we obtain for $\lambda \in [0,\alpha)$:
$$s_i^\alpha(\lambda v_I) = \begin{cases}
                      \alpha v_{\{0,i\}} + \lambda v_{I \mid n+1}, &\textnormal{ if } I_0 \nsubseteq I \\
		      (\alpha-\lambda) v_{\{0,i\}} + \lambda v_{I \mid n+1}, &\textnormal{ otherwise}
                     \end{cases}.$$  

Now for an arbitrary $v = \sum \lambda_I v_I \in U_\alpha$ (where we can assume that all the $v_I$ with $\lambda_I \neq 0$ lie in the same maximal cone in $\mn$) we have 
$$s_i^\alpha(v) = \sum \lambda_I v_{I \mid n+1} + \underbrace{(\alpha - \sum_{I_0 \subseteq I} \lambda_I)}_{> 0} v_{\{0,i\}}.$$
In particular this is a vector in a leaf of the fibre of $v$ which as a set can be described as $\{\sum \lambda_I v_{I \mid n+1} + \gamma v_{\{0,i\}}, \gamma \geq 0\}$, and for different $i$ this marks a different leaf. Also it is clear that for different $\alpha, \alpha'$ and $v \in U_\alpha \cap U_{\alpha'}$, $s_i^\alpha$ and $s_i^{\alpha'}$ mark the same leaf. Hence the $s_i^\alpha$ define a tropical marking.
\end{proof}
\end{proposition}

We will now prove that any two markings on the forgetful map only differ by a permutation on $\{1,\dots,n\}$.

\begin{proposition}\label{families_prop_equivalence}
For any two families of tropical curves of the form $$(\mnp \stackrel{\ft_0}{\to} \mn, (s_i^{\theta})), (\mnp \stackrel{\ft_0}{\to} \mn, (r_i^{\zeta})),$$ there exist isomorphisms $\phi: \mn \to \mn$ and $\psi: \mnp \to \mnp$ such that $\ft_0 \circ \psi = \phi \circ \ft_0$ and such that for any $b$ in $\mn$, $\psi$ identifies equally marked leaves of $\ft_0^*(b)$ and $\ft_0^*(\phi(b))$ in the two families. Furthermore, $\phi,\psi$ are induced by permutations on the coordinates of $\R^{\binom{n}{2}}$ and $\R^{\binom{n+1}{2}}$ respectively.
\begin{proof}
 We can assume without restriction that both markings $(s_i^{\theta}), (r_i^{\theta})$ are defined on the same open subsets $U_\theta$. Since they are tropical markings, if we choose $\theta$ such that $0 \in U_\theta$, we must have for all $i$ that $$s_i^{\theta}(0) = \lambda_i^\theta v_{\{0,\sigma_1(i)\}};\; r_i^{\theta}(0) = \rho_i^\theta v_{\{0,\sigma_2(i)\}}$$ 
for some permutations $\sigma_1,\sigma_2\in \Sn_n , \lambda_i^\theta,\rho_i^\theta > 0.$ Note that by definition of a marking, $\sigma_1,\sigma_2$ are independent of the choice of $\theta$.

We can extend $\sigma_1,\sigma_2$ to bijections $\bar{\sigma}_1,\bar{\sigma}_2$ on $\{0,1,\dots ,n\}$ by setting $\bar{\sigma}_1(0) = \bar{\sigma}_2(0) = 0$. 
These bijections induce automorphisms of $\R^{\binom{n+1}{2}}$ and $\R^{\binom{n}{2}}$ given by
$$e_{\{i,j\}} \mapsto e_{\{(\bar{\sigma}_2 \circ \bar{\sigma}_1^{-1})(i), \bar{\sigma}_2\circ \bar{\sigma}_1^{-1})(j)\}},$$
which map $\Im(\phi)$ to $\Im(\phi)$  and thus give rise to automorphisms
\[
 \psi: \mnp \to \mnp,\ \ \ \
 \phi: \mn \to \mn.
\]
Since the $0$-mark which is discarded by $\ft_0$ is not affected by $\sigma_1,\sigma_2$ we conclude that $\ft_0 \circ \phi = \psi \circ \ft_0$. We will now prove compatibility with markings for ray vectors $v_I$:

Let $v_I \in U_\zeta \subseteq |\mn|$ with $i \notin I$ and assume $\phi^{-1}(v_I) = v_{(\sigma_1 \circ \sigma_2^{-1})(I)} \in U_\theta \subseteq |\mn|$. Then we have
$$r_i^\zeta(v_I) = v_{I \mid n+1} + \lambda \cdot v_{\{0, \sigma_2(i)\}}$$
for some $\lambda$ and
\begin{align*}
 (\psi \circ s_i^\theta \circ \phi^{-1})(v_I) 
	    &= (\psi \circ s_i^\theta)(v_ {(\sigma_1 \circ \sigma_2^{-1})(I)}) \\
	    &= \phi(v_{(\sigma_1 \circ \sigma_2^{-1})(I) \mid n+1} + \rho \cdot v_{\{0,\sigma_1(i)\}}) \textnormal{ for some } \rho\\
	    &= v_{(\sigma_2 \circ \sigma_1^{-1} \circ \sigma_1 \circ \sigma_2^{-1})(I) \mid n+1} + \rho \cdot v_{\{0, (\sigma_2 \circ \sigma_1^{-1} \circ \sigma_1)(i)\}} \\
	    &= v_{I \mid n+1} + \rho \cdot v_{\{0,\sigma_2(i)\}}
\end{align*}
which lies on the same leaf as $r_i^\zeta(v_I)$. For an arbitrary vector $v = \sum \alpha_I v_I$ the same argument can be applied by linearity of $\phi$.
\end{proof}
\end{proposition}


We now apply our theory to assign a family of $n$-marked curves to each morphism from a smooth cycle to $\mn$. Let us first introduce some notation.

\begin{notation}
\label{Xf}
Let $X$ be a smooth variety and $f:X\rightarrow\mn$ a morphism. Then we denote by $X^f$ the fibre product \[X^f:=X\times_{\mn} \mnp\in\Zy_{\dim X+1}(X\times\mnp).\]
\end{notation}

We conclude in the following corollary that the projection $\pi_X:X^f\rightarrow X$ is a family of $n$-marked curves.

\begin{corollary}\label{families_cor_pullback}
For each morphism of smooth varieties $X \stackrel{f}{\to} \mn$, we obtain a family of $n$-marked rational curves as
$$(X^f  \stackrel{\pi_X}{\to} X, t_i^\alpha),$$
where $t_i^\alpha: f^{-1}(U_\alpha) \to X^f, x \mapsto (x, s_i^\alpha \circ f (x))$ and $s_i^\alpha$ is the marking on the universal family from proposition \ref{markingmn}.
\end{corollary}
\begin{proof}
The cycle $X^f$ is a tropical variety by proposition \ref{weight1} and $\pi_X$ is locally surjective by proposition \ref{thirdaxiompi}. Each fibre $\pi_X^* (p)=\{p\}\times \ft^*(f(p))$ is a smooth rational curve with $n$ leaves by propositions \ref{fibreXf} and \ref{forgetful}. It is obvious that $\pi_X$ satisfies the second prefamily axiom and that $t_i^\alpha$ is indeed a marking.
\end{proof}

\begin{example}
We finish the section by introducing an alternative way of constructing the moduli spaces $\mn$. Let us briefly recall the notion of tropical modifications introduced in \cite[section 3.3]{mikh} and used in this construction. The modification of a cycle $X$ in $V$ along the rational function $\varphi$ on $X$ is the cycle $$\Gamma_{\varphi,X}:=\max\{\pi_X^*\varphi,y\}\cdot X\times \R,$$ where $\pi_X: X\times \R \rightarrow X$ is the projection to $X$ and $y$ is the coordinate describing $\R$. In other words, the modification is the graph of $\varphi$ made balanced by adding cells in the direction $(0,-1)\in V\times \R$. If $Y=\varphi\cdot X$ one says, by slight abuse of notation, that $\Gamma_{\varphi,X}$ is the modification of $X$ along $Y$.

We prove in the following proposition that $\mathcal{M}_{n+2}$ is the modification of
the fibre product $\mnp\times_{\mn} \mnp$ along the codimension $1$ subcycle $\Delta_{\mnp}$. This leads to an alternative procedure of constructing $\mn$ which is of course very similar to construction of the classical moduli spaces $\overline{M}_{0,n}$ in \cite[section 1.4]{KV}.

Our proof uses the fact that $\mn$ is isomorphic to $\trop(K_{n-1})/L$ and the connection between tropical modifications and the matroid-theoretic concepts of deletions and contractions.
If $S$ is the set of flats of a matroid $M$ and $e\in E(M)$, then the set of flats of the deletion $M\setminus e$ is $\{F\setminus \{e\}:F\in S\}$, whereas the set of flats of the contraction $M/e$ is $\{F: F\cup \{e\} \in S\}$. In the case that $e$ is not a coloop, the matroid variety $\trop(M)$ is the modification of $\trop(M\setminus e)$ along $\trop(M / e)$ 
(cf.\ \cite[proposition 2.24]{shaw} or \cite[proposition 3.10]{francoisrau}).
\end{example}

\begin{proposition}
Let $\pi_{\mnp}:\mnp^{\ft}\rightarrow\mnp$ be the family of $n$-marked curves induced by the forgetful map $\ft:\mnp\rightarrow \mn$. Then the modification of $\mnp^{\ft}$ along its codimension $1$ subcycle $\Delta_{\mnp}$ is the moduli space of $(n+2)$-marked abstract rational curves $\mathcal{M}_{n+2}$.
\end{proposition}
\begin{proof}
Let $K_{n+1}$ be the complete graph on the vertex set $\{0,1,\ldots,n\}$. It suffices to prove that $\mnp^{\ft}$ is isomorphic to $\trop(M(K_{n+1})\setminus (0,n))/L$ and that $\Delta_{\mnp}$ is isomorphic to $\trop(M(K_{n+1})/ (0,n))/L$, where $L=\R\cdot (1,\ldots, 1)$ and $(0,n)$ denotes the edge between $0$ and $n$. We consider the injective linear map
\begin{eqnarray*}
f:  \R^{\binom{n+1}{2}-1}&\rightarrow &\R^{\binom{n}{2}}\times\R^{\binom{n}{2}} \\   (x_{i,j})_{0\leq i<j\leq n: (i,j)\neq (0,n)} &\mapsto & ((x_{i,j})_{0\leq i<j\leq n-1},(x_{i,j})_{1\leq i<j\leq n}) .
\end{eqnarray*}
Let $\pi_0, \pi_n:\R^{\binom{n}{2}}\rightarrow\R^{\binom{n-1}{2}}$ be the projections that forget all coordinates $x_{0,i}$ and $x_{i,n}$ respectively; in other words, they describe the forgetful maps $\ft_0,\ft_n$. Let $\pi_{(0,n)}:\R^{\binom{n+1}{2}}\rightarrow\R^{\binom{n+1}{2}-1}$ be the projection which forgets the coordinate $x_{0,n}$. With these notations we obviously have $f\circ\pi_{(0,n)}=(\pi_n,\pi_0)$. Thus we obtain
\[f_*\trop(M(K_{n+1})\setminus (0,n)) = f_*{\pi_{(0,n)}}_*\trop(K_{n+1}) = (\pi_n,\pi_0)_* \trop(K_{n+1}).\]
Therefore, we can conclude that
\[ |f_*\trop(M(K_{n+1})\setminus (0,n))|= \{ (x,y)\in\trop(K_n)\times\trop(K_n): \pi_0(x)=\pi_n(y)\}.\]
Here the first complete graph $K_n$ has vertex set $\{0,1,\ldots,n-1\}$, whereas the second has vertex set $\{1,\ldots,n\}$. As all occurring weights are $1$, it follows by theorem \ref{fibreproduct} that $f_*\trop(M(K_{n+1})\setminus (0,n))/L$ is isomorphic to $\mnp^{\ft}$.\\
In order to prove the second part we notice that $\trop(M(K_{n+1})/ (0,n))/L$ and $\Delta_{\mnp}$ are both matroid varieties modulo lineality spaces and have the same dimension. Therefore, it suffices to show that for every flat of $M(K_{n+1})/(0,n)$, $f(V_F)$ is in the diagonal of $\trop(K_n)\times\trop(K_n)$ after identifying the coordinates $x_{0,i}$ of the first $\R^{\binom{n}{2}}$ with the coordinates $x_{i,n}$ of the second to obtain the same set of coordinates in both factors. If $F$ is a flat of $M(K_{n+1})/(0,n)$, then $F\cup \{(0,n)\}$ is a flat in $K_{n+1}$; but this implies that $(0,i)\in F$ if and only if $(i,n)\in F$. Hence $f(V_F)$ lies in the diagonal.
\end{proof}

\section{The fibre morphism}\label{section_fibre_morphism}
\label{fibremorphism}

We now want to construct a morphism into $\mn$ for a given family $T \stackrel{g}{\to} B$ (we will omit the marking to make the notation more concise). It is actually already clear what this map should look like: It should map each $b$ in $B$ to the point in $\mn$ that represents the fibre over $b$. For the pull-back family $X^f$ defined above this gives us back the map $f$. For an arbitrary family however, it is not even clear that it is a morphism. In fact, we will only show that it is a so-called \emph{pseudo-morphism} and then use the fact that $B$ is smooth to deduce that it is a morphism.

\begin{definition}[The fibre morphism]
\label{deffibremorph}
For a family $T \stackrel{g}{\to} B$ we define a map 
$$d_g: B \to \R^{\binom{n}{2}}: b \mapsto (\dist_{k,l}(g^*(b)))_{k < l},$$
where the length of the path from leaf $k$ to leaf $l$ on the fibre is determined in the following way: The length of a bounded edge $E = \conv\{p,q\}$ is defined to be the positive real number $\alpha$ such that $q = p + \alpha \cdot v$, where $v$ is the primitive lattice vector generating that edge.

We define $\varphi_g := q_n \circ d_g: B \to \mn$, where $q_n:\R^{\binom{n}{2}}\rightarrow\R^{\binom{n}{2}}/\Im(\phi_n)$ is the quotient map and $\phi_n$ maps $x\in\R^n$ to $(x_i+x_j)_{i<j}$.
\end{definition}

As mentioned above, we will not be able to prove directly that $\varphi_g$ is a morphism. But we can show that, in addition to being piecewise linear, it respects the balancing equations of $B$. Let us make this precise:

\begin{definition}[Pseudo-morphism]
 A map $f: X \to Y$ of tropical cycles is called a \emph{pseudo-morphism} if there is a polyhedral structure $\curlyx$ of $X$  such that:
\begin{enumerate}
 \item $f_{\mid \tau}$ is integer affine linear for each $\tau \in \curlyx$
 \item $f$ respects the balancing equations of $X$, i.e.\ for each $\tau \in \curlyx^{(\dim X- 1)}$ if $\bar{f}$ denotes the induced piecewise affine linear map on $\Star_X(\tau)$ (cf.\ \cite[section 1.2.3]{disshannes}), we have $$\sum_{\sigma > \tau} \omega_X(\sigma) \bar{f}(u_{\sigma/\tau}) = 0 \in V / V_{f(\tau)}.$$ 
\end{enumerate}
As for a morphism, we denote by $\lambda_{f \mid \tau}$ the linear part of $f$ on $\tau$.
\end{definition}

\begin{remark}
 We can reformulate the second condition as follows: If we choose a $v_\sigma \in \sigma$ for each $\sigma > \tau$ and $p_0,...,p_d \in \tau$ a basis of $V_\tau$ such that $\overline{v_\sigma - p_0} = u_{\sigma/\tau}$ and $\sum_{\sigma > \tau} \omega_X(\sigma) (v_\sigma - p_0) = \sum_{i=1}^d \alpha_i (p_i - p_0)$ with $\alpha_1,...\alpha_d \in \R$, then 
$$\sum_{\sigma > \tau} \omega_X(\sigma)(f(v_\sigma) - f(p_0)) = \sum_{i=1}^d \alpha_i (f(p_i) - f(p_0)).$$
Note that it suffices to check this condition for a single choice of $v_\sigma, p_0,...p_d$, since any other choice would only differ by elements from $V_\tau$, on which $f$ is affine linear. It is also clear that $f$ satisfies the above properties on any refinement of $\curlyx$ if and only if it does so for $\curlyx$.
\end{remark}

\begin{proposition}\label{fibremorphism_prop_smoothpseudo}
Let $X$ be a smooth tropical variety, $Y$ any tropical cycle and $f: X \to Y$ a pseudo-morphism. Then $f$ is a morphism.
\begin{proof}
It suffices to prove that each piecewise linear pseudo-morphism $f:\trop(M)\rightarrow Y$ from a matroid variety to a fan cycle is a linear map because being a morphism is a local property and we can lift any pseudo-morphism $\trop(M)/L\rightarrow Y$ to a pseudo-morphism $\trop(M)\rightarrow Y$. By deleting parallel elements we can assume that one element subsets of the ground set $E(M)$ are flats of $M$. It is easy to see that $f$ must be a pseudo-morphism with respect to the fan structure $\mf(M)$. Now we show by induction on the rank of the flats that for all flats $F$ we have $f(V_F)=\sum_{i\in F} f(V_{\{i\}})$. As the vectors $V_{\{i\}}$ are linearly independent this implies that $f$ is linear. Let $F$ be a flat of rank $r$. We choose a chain of flats of the form $\F=(\emptyset\subsetneq F_1\subsetneq\ldots\subsetneq F_{r-2}\subsetneq F \subsetneq F_{r+1}\subsetneq\ldots\subsetneq F_{\rank(M)}=E(M))$, with $\rank(F_i)=i$. The fact that $f$ is a pseudo-morphism translates the balancing condition around the facet $\F$ in $\mf(M)$ into
\[ \sum_{F_{r-2}\subsetneq G \subsetneq F \text{ flat } } f(V_G) = f(V_F)+ (\betrag{\{G: F_{r-2}\subsetneq G \subsetneq F \text{ flat }\}}-1)\cdot f(V_{F_{r-2}}).\]
Now the induction hypothesis for the flats $G,F_{r-2}$ implies that, as required, $f(V_F)=\sum_{i\in F} f(V_{\{i\}})$.
\end{proof}
\end{proposition}

\begin{proposition}\label{fibremorphism_prop_morphism}
 For any family $T \stackrel{g}{\to} B$, the map $\varphi_g: B \to \mn$ is a pseudo-morphism.
\end{proposition}

Before we give a proof of this proposition we use it to prove our main theorem.

\begin{theorem}
\label{maintheorem}
 For any smooth variety $B$, we have a bijection
\begin{align*}
 \left\{\substack{\textnormal{Families } (T \stackrel{g}{\to} B,r_i^\theta) \\ \textnormal{of $n$-marked tropical curves } \\ \textnormal{modulo equivalence}}\right\} &
      \stackrel{1:1}{\longleftrightarrow}
  \left\{\substack{\textnormal{Morphisms} \\ \\ f: B \to \mn}\right\}\\
 (T \stackrel{g}{\to} B,r_i^\theta) &\mapsto \varphi_g \\
 (B^f \stackrel{\pi_B}{\to} B, (\id \times (s_i^\alpha \circ f))) & \mapsfrom f ,
\end{align*}
where $\varphi_g:B\rightarrow\mn$ is the morphism constructed in definition \ref{deffibremorph}, $B^f$ is the tropical subvariety of $B\times\mnp$ introduced in definition \ref{Xf}, $\pi_B: B^f\rightarrow B$ is the projection to $B$, and $s_i^\alpha, i=1,\ldots,n$ is the tropical marking of the forgetful map described in proposition \ref{markingmn}.
\begin{proof}
 We have already shown in corollary \ref{families_cor_pullback} and proposition \ref{fibremorphism_prop_morphism} that these maps are well-defined. It is obvious that they are inverse to each other.
\end{proof}
\end{theorem}

\begin{corollary}
 The tropical variety $\mn$ is a fine moduli space for the contravariant functor $F$ from the category of smooth tropical varieties into the category of sets, defined by
\begin{align*}
F: \textnormal{Obj}((SmoothTrop)) &\to \textnormal{Obj}((Sets))\\
 B &\mapsto \left\{\substack{\textnormal{Families } (T \stackrel{g}{\to} B,r_i^\theta) \\ \textnormal{of $n$-marked tropical curves } \\ \textnormal{modulo equivalence}}\right\}\\
\textnormal{Mor}((SmoothTrop)) &\to \textnormal{Mor}((Sets))\\
(B \stackrel{f}{\to} B') &\mapsto f^*,
\end{align*}
where 
\begin{align*}
f^*: \{T' \to B'\} &\to \{T \to B\} \\
      (T' \to B') &\mapsto (B^{\varphi_{g'} \circ f} \to B)
\end{align*}
is the pull-back of families induced by composing $f$ with the fibre morphism and constructing the corresponding family.
\end{corollary}

The rest of this section is dedicated to proving proposition \ref{fibremorphism_prop_morphism}. For all the following proofs, we will assume that $\curlyt$ and $\curlyb$ are polyhedral structures of $T$ and $B$ satisfying $\curlyb = \{g(\sigma), \sigma \in \curlyt\}$. This is possible by \cite[lemma 1.3.4]{disshannes}.

\begin{lemma}\label{fibremorphism_lemma_homeo}
 Fibres over the relative interior of a cell $\tau$ in $\curlyb$ have the same combinatorial type. More precisely: For each $\tau \in \curlyb, b \in \tau, b' \in \relint \tau$, there exists a piecewise linear, continuous and surjective map
$$t_{b',b}: g^*(b') \to g^*(b)$$
for which the following holds:
\begin{enumerate}
 \item If $b,b' \in \relint(\tau), t_{b',b}$ is a homeomorphism
 \item If $l_i(b), l_i(b')$ denote the $i$-th leafs of the respective fibre, then $$t_{b',b}(l_i(b')) = l_i(b).$$
 \item On each edge $e$ of $g^*(b')$, $t_{b',b}$ is affine linear and $e$ is either mapped bijectively onto an edge with the same primitive direction vector or to a single vertex. In particular, vertices are mapped to vertices.
 \item If $e_1,e_2$ are two different edges of $g^*(b')$, then 
$$\betrag{t_{b',b}(e_1) \cap t_{b',b}(e_2)} \leq 1.$$
 \item For each $\sigma \in \curlyt$ such that $g(\sigma) = \tau$, we have $$t_{b',b}(\betrag{ g^*(b')} \cap \sigma) \subseteq \sigma.$$
\end{enumerate}
\begin{proof}
 Let $\sigma \in \curlyt$ such that $g(\sigma) = \tau$ and denote by $C_b := \betrag{g^*(b)} \cap \sigma, C_{b'} := \betrag{g^*(b')} \cap \sigma$. Similarly, if $b_\lambda := b + \lambda(b'-b), \lambda \in [0,1]$, we denote by $C_{b_\lambda}$ its fibre in $\sigma$. 

If $\dim \sigma = \dim \tau$, then $f_{\mid \sigma}$ is injective and $C_b, C_{b'}$ are single points.

If $\dim \sigma = \dim \tau +1$, then $C_{b'}$ must be a line segment and $C_b$ is either a parallel line segment or a point. Furthermore, $C_b$ is unbounded, if and only if $C_{b'}$ is unbounded (in the sense that it intersects $\partial \sigma$ in only one point). Indeed, assume $C_{b'}$ to be unbounded. Then $C_{b'} = \{ x + \alpha \cdot v; \alpha \geq 0\}$ for some $x$ and $v$ in $\R^n$. Now let $p,q$ be distinct points in $C_b$. Then 
\begin{align*}
 \sigma &\ni (1-\lambda)p + \lambda(q+ \alpha v) \;\textnormal{ for all } \lambda \in [0,1], \alpha \geq 0\\
	&= ((1-\lambda)p + \lambda q) + \alpha \lambda v \in C_{b_\lambda}
\end{align*}
due to convexity of $\sigma$. Hence $C_{b_\lambda}$ is unbounded for all $\lambda > 0$ and since $g$ is a continuous map, $C_b$ must be unbounded as well. The other implication can be proven analogously.

This gives us a canonical affine linear map $t_{b',b}^\sigma: C_{b'} \to C_b$ on each cell $\sigma$ such that $g(\sigma) = \tau$. We can obviously glue these together to a piecewise affine linear map $t_{b',b}: g^*(b') \to g^*(b)$ (we will shorten this to $t$ here for simplicity).

If $b \in \relint (\tau)$ as well, we see that $g_{\mid \sigma}^{-1}(\relint (\tau)) \subseteq \relint (\sigma)$ for any $\sigma$ on which $g$ is not injective, so $C_b$, $C_{b'}$ are both line segments and $t$ becomes a homeomorphism. Obviously $t$ is affine linear on each edge of $g^*(b')$. Hence, if $t_{\mid e}$ is not injective for some edge $e$, it must be constant. Since $t_{\mid \sigma}$ preserves edge directions, so does $t_{\mid e}$. Furthermore, if any vertex $v$ were to be mapped onto the interior of an edge $e'$ of $g^*(b)$, then all edges adjacent to $v$ would have to be mapped to $e$ as well. But two different edges at $v$ have linearly independent direction vectors, so they must live in different cells of $\curlyt$. Hence their images can only intersect in at most one point.

Finally we see that a leaf can obviously only be mapped to a leaf with the same direction vector. Affine linearity of $s_i^\theta$ implies that they must be marked by the same map $s_i$.
\end{proof}
\end{lemma}

\begin{proposition}\label{dgaffinelinear}
  The map $d_g$ of definition \ref{deffibremorph} is integer affine linear on each $\tau\in\curlyb$.
\begin{proof}
We first show that $d_g$ is affine linear on each cell: Since $\tau\in\curlyb$ is closed and convex, it suffices to show that $d_g$ is affine linear on any line segment $\conv\{b,b'\} \subseteq \tau$, where $b \in \tau$ and $b' \in \relint (\tau)$.

 Let $\sigma \in \curlyt$ such that $g(\sigma) = \tau$ and $\dim \sigma = \dim \tau +1$ and both fibres $C_b, C_{b'}$ are bounded (see proof of \ref{fibremorphism_lemma_homeo}, obviously only these fibres are relevant for the distance map $d_g$). Then the map $d_g^\sigma: \conv\{b,b'\} \to \R$, which assigns to $b_\lambda := b + \lambda(b'-b) \in \conv\{b,b\}$ the length of its fibre in $\sigma$, is affine linear since $g_{\mid \sigma}^{-1}(\conv\{b,b'\})$ is a polyhedron.

Denote by $G_{b_\lambda}(k,l)$ the set of all cells in $\curlyt$ of dimension $(\dim \tau + 1)$ such that $g_{\mid \sigma}^{-1}(b_\lambda)$ is contained in the path from $k$ to $l$ in the curve $g^*(b_\lambda)$. Then we have
$$\dist_{k,l}(g^*(b_\lambda)) = \sum_{\sigma \in G_{b_\lambda}(k,l)} d_g^\sigma(b_\lambda).$$
Since we know that $d_g^\sigma$ is affine linear, it suffices to show that $G_{b_\lambda}(k,l) = G_{b_\rho}(k,l)$ for all $\lambda,\rho \in [0,1]$, which immediately follows from the fact that the map $t_{\lambda,\rho}$ identifies equally marked leaves and hence edges lying on the same path.

It remains to show that $d_g$ is an integer map: We want to show that for $b,b' \in \tau$ such that $b-b' \in \Lambda_\tau$, we have $d_g(b') - d_g(b) \in \Z^{\binom{n}{2}}$. Choose $\sigma$ such that the fibre of $b'$ in $\sigma$ is a bounded line segment. It is easy to see that we must have two endpoints $p,q$ of both fibres lying in the same face $\sigma' < \sigma$, and hence in the same hypersurface of $V_\sigma$ which is defined by an integral equation
$$h(x) = \alpha;\; h \in \Lambda_\sigma^\vee, \alpha \in \R.$$
By surjectivity of $\bar{\lambda}_{g \mid \sigma}: \Lambda_\sigma \to \Lambda_\tau$, we have
$$\Lambda_\sigma \cong \Lambda_\tau \times \gnrt{v}_\Z$$
for some primitive integral vector $v$ (which generates $\ker \lambda_{g\mid\sigma}$).

Under this isomorphism we write the coordinates of $p, q$ and $h$ as 
\begin{align*}
 p &= (p_1,\dots ,p_k,p_v) \\
 q &= (q_1,\dots ,q_k,q_v) \\
 h(x_1,\dots ,x_k,x_v) &= h_1x_1 + \dots + h_kx_k + h_v x_v,
\end{align*}
where $p_i-q_i \in \Z$ for $i = 1,\dots, k$, $h_j \in \Z$ for all $j$ and $h_v \neq 0$ (since otherwise $\lambda_g$ would not be injective on the corresponding hypersurface). Now the identity $h(p-q) = 0$ transforms into
\begin{align*}
 0 &= \sum_{i=1}^k (q_i - p_i) h_i + (q_v - p_v) h_v\\
   &= \underbrace{\sum_{i=1}^k (b'-b)_{i} h_i}_{\in \Z} + (q_v - p_v) \underbrace{h_v}_{\in \Z}.
\end{align*}
Hence $q_v - p_v \in \Q$ and $q-p \in \Lambda_\sigma \otimes_\Z \Q$.

So there exists a minimal $k \in \N$ such that $k \cdot (q-p) \in \Lambda_\sigma$. In particular, $k \cdot (q-p)$ is primitive. Assume $k > 1$. Then $\bar{\lambda}_g(k \cdot (q-p)) = k \cdot (b' - b)$. By surjectivity of $\bar{\lambda}_g$, there exists an $a \in \Lambda_{\sigma'}$ such that $\bar{\lambda}_g(a) = b'-b$. (Note that we cannot use lemma \ref{lambdasurjective} here since $\lambda_{\mid \sigma'}$ is injective.) This implies $\bar{\lambda}_g(k \cdot a) = \bar{\lambda}_g(k \cdot (q-p))$. Since $\bar{\lambda}_g$ is injective on $\Lambda_{\sigma'}$, we must have $k \cdot a = k \cdot (q-p)$, which is a contradiction since the latter is primitive. Hence $k = 1$ and $q-p \in \Lambda_\sigma$.

Finally we obtain 
$$\Lambda_\sigma \ni (q'-p') - (q-p) = (d_g^\sigma(b') - d_g^\sigma(b)) \cdot v.$$
Hence, since $v$ is primitive, $d_g^\sigma(b') - d_g^\sigma(b) \in \Z$ and the same follows for $d_g(b') - d_g(b)$.
\end{proof}
\end{proposition}

Before we can prove that $\varphi_g$ is a pseudo-morphism, we need to fix a few notations:

\begin{notation}\label{fibremorphism_notation}
\indent\par
\begin{itemize}
 \item Let $\tau\in\curlyb^{(\dim \curlyb-1)}$. Choose $p_0,p_1,\dots,p_d \in \relint (\tau)$ such that $\{p_i - p_0; i = 1,\dots ,d\}$ is a basis of $V_\tau$. Furthermore, for each $\sigma > \tau$, choose a point $v_\sigma \in \relint (\sigma)$ such that $v_\sigma - p_0$ is a representative of $u_{\sigma / \tau}$. We can assume that this is possible since there always exist  $v_\sigma \in \relint (\sigma), q_\sigma \in \Q$ such that $v_\sigma - p_0 = q_\sigma \cdot u_{\sigma / \tau}$ modulo $V_\tau$. We can then make our choice such that $q_\sigma = q_{\sigma'} =: q$ for all $\sigma, \sigma' > \tau$, so 
  $$\sum_{\sigma > \tau} u_{\sigma/\tau} = \frac{1}{q}\sum_{\sigma > \tau} (v_\sigma - p_0).$$
Hence the left hand side is in $V_\tau$ if and only if the right hand side is.

So we obtain that $$\sum_{\sigma > \tau} (v_\sigma - p_0) = \sum_{j=1}^d \alpha_j (p_j - p_0)$$
for some $\alpha_j \in \R$.
\item Lemma \ref{fibremorphism_lemma_homeo} justifies the following definitions: We fix $k,l \in [n]$.
\begin{itemize}
\item Denote by $q_1,\dots ,q_r \in T$ the vertices of the fibre $g^*(p_0)$ which lie on the path from $k$ to $l$.
\item The fibre of $p_j$ has the same combinatorial type as $g^*(p_0)$, so for $j = 1,\dots, d$, denote by $q_i^{(j)}, i = 1,\dots ,r$ the $i$-th vertex in the fibre of $p_j$.
\item Let $\sigma > \tau$. The preimage of $q_i$ under $t_{v_\sigma,p_0}$ contains a certain number of vertices lying on the path from $k$ to $l$, the first and last of which we denote by $q_{i,k}^\sigma$ and $q_{i,l}^\sigma$ respectively.
\item Let $w_i, i = 1,\dots ,r-1$ be the primitive direction vector of the bounded edge from $q_i$ to $q_{i+1}$. We define the lengths $e_i, e_i^{(j)}, e_i^\sigma > 0$ of the corresponding edges via:
  \begin{align*}
   q_{i+1} &= q_i + e_i \cdot w_i,\\
   q_{i+1}^{(j)} &= q_i^{(j)} + e_i^{(j)} \cdot w_i,\\
   q_{i+1,k}^\sigma &= q_{i,l}^\sigma + e_i^\sigma \cdot w_i.
  \end{align*}
\item In addition we fix $w_0 := -v_k, w_r := v_l$, where $v_k$ and $v_l$ are the primitive direction vectors of the leaves marked $k$ and $l$.
\item For $i = 1,\dots, r$, denote by $e_{i,t}^\sigma, t=1,\dots ,r_{i,\sigma}$ the length of the edges  on the path from $q_{i,k}^\sigma$ to $q_{i,l}^\sigma$.
\end{itemize}
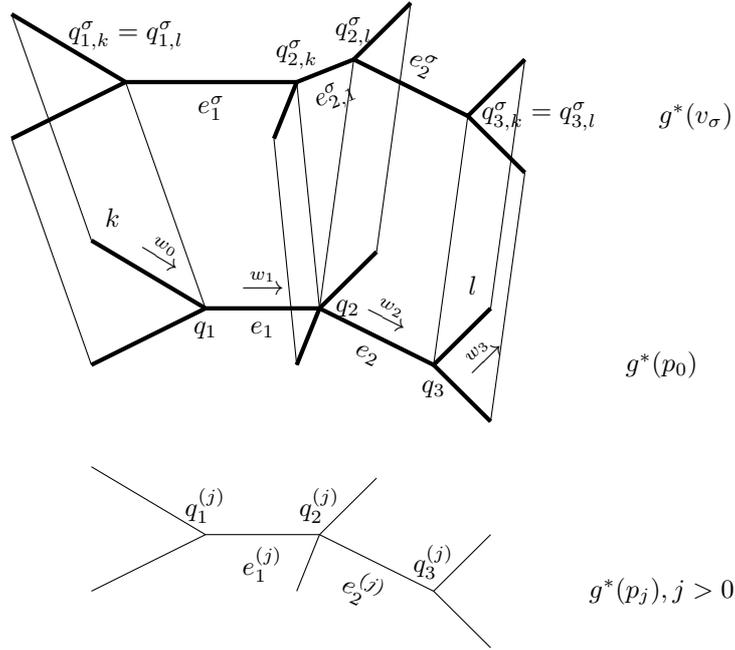
\begin{figure}[h]
 \centering
 \begin{tikzpicture}[scale = 1.5]
  \node (a0) at (0,0.6) {};
  \node (a1) at (0,-0.5) {};
  \node (q1) at (1,0) {};
  \node (q2) at (2,0) {};
  \node (a2) at (1.8,-0.5) {};
  \node (a3) at (2.5,0.5) {};
  \node (q3) at (3,-0.5) {};
  \node (a4) at (3.5,-1) {};
  \node (a5) at (3.5,0) {};

  \node (diff) at (0,-2) {};
  \node (disig) at (-0.2,2) {};
  \node (delta) at (0.5,0.2) {};
  \node (eps) at (-0.5,0) {};

  \node (tdiff) at (2,0) {};

  \node (sigmafibre) at ($(q3) + (disig) + (delta) + (tdiff)$) {$g^*(v_\sigma)$};
  \node (taufibre) at ($(q3) + (tdiff)$) {$g^*(p_0)$};
  \node (pifibre) at ($(q3) + (diff) + (tdiff)$) {$g^*(p_j), j > 0$};

 \begin{scope}[ultra thick]
  \draw (a0.center) node[above right = 1pt]{$k$} -- 
	(q1.center) 
	    node[below = 1pt]{$q_1$} 
	    node[midway, sloped, above = 1pt]{$\stackrel{w_0}{\longrightarrow}$}-- 
	(q2.center) 
	    node[midway, sloped, above = 1pt]{$\stackrel{w_1}{\longrightarrow}$}
	    node[midway, sloped, below = 1pt]{$e_1$} -- 
	(q3.center) node[midway, sloped,above = 1pt]{$\stackrel{w_2}{\longrightarrow}$} 
	    node[midway, sloped, below = 1pt]{$e_2$} -- 
	(a5.center) 
	    node[above left = 1pt]{$l$}
	    node[midway, sloped, below = 1pt]{$\stackrel{w_3}{\longrightarrow}$};
  \draw (a1.center) -- (q1.center);
  \draw (q2.center) node[right = 2pt]{$q_2$} -- (a2.center);
  \draw (q2.center) -- (a3.center);
  \draw (q3.center) node[below = 2pt]{$q_3$} -- (a4.center);

  \draw ($(a0) + (disig) + (eps)$.center)  -- 
	($(q1) + (disig) + (eps)$.center) 
	    node[above = 8pt]{$q_{1,k}^\sigma = q_{1,l}^\sigma$}-- 
	($(q2) + (disig)$.center) 
	    node[above = 1pt]{$q_{2,k}^\sigma$} 
	    node[midway,sloped,below = 1pt]{$e_1^\sigma$}-- 
	($(q2)+(disig) + (delta)$) 
	    node[above = 1pt]{$q_{2,l}^\sigma$}
	    node[midway, sloped, below = 1pt]{$e_{2,1}^\sigma$}-- 
	($(q3) + (disig) + (delta)$.center) 
	    node[right = 1pt]{$q_{3,k}^\sigma = q_{3,l}^\sigma$}
	    node[midway, sloped, above= 1pt]{$e_2^\sigma$}-- 
	($(a5)+(disig)+ (delta)$.center);
  \draw ($(a1) + (disig) + (eps)$.center) -- ($(q1) + (disig) + (eps)$.center);
  \draw ($(q2) + (disig)$.center) -- ($(a2) + (disig)$.center);
  \draw ($(q2) + (disig) + (delta)$.center) -- ($(a3) + (disig) + (delta)$.center);
  \draw ($(q3) + (disig) + (delta)$.center) -- ($(a4) + (disig) + (delta)$.center);
 \end{scope}
  \draw ($(a0) + (diff)$.center)  -- 
	($(q1) + (diff)$.center) 
	    node[above = 1pt]{$q_1^{(j)}$} -- 
	($(q2) + (diff)$.center) 
	    node[above = 1pt]{$q_2^{(j)}$} 
	    node[midway, sloped,below= 1pt]{$e_1^{(j)}$}-- 
	($(q3) + (diff)$.center) 
	    node[above = 1pt]{$q_3^{(j)}$} 
	    node[midway,sloped,below = 1pt]{$e_2^{(j)}$}-- 
	($(a5)+(diff)$.center);
  \draw ($(a1) + (diff)$.center) -- ($(q1) + (diff)$.center);
  \draw ($(q2) + (diff)$.center) -- ($(a2) + (diff)$.center);
  \draw ($(q2) + (diff)$.center) -- ($(a3) + (diff)$.center);
  \draw ($(q3) + (diff)$.center) -- ($(a4) + (diff)$.center);

  \draw (q1.center) -- ($(q1) + (disig) + (eps)$.center);
  \draw (q2.center) -- ($(q2) + (disig)$.center);
  \draw (q2.center) -- ($(q2) + (disig) + (delta)$.center);
  \draw (q3.center) -- ($(q3) + (disig) + (delta)$.center);
  \draw (a0.center) -- ($(a0) + (disig) + (eps)$.center);
  \draw (a1.center) -- ($(a1) + (disig) + (eps)$.center);
  \draw (a2.center) -- ($(a2) + (disig)$.center);
  \draw (a3.center) -- ($(a3) + (disig) + (delta)$.center);
  \draw (a4.center) -- ($(a4) + (disig) + (delta)$.center);
  \draw (a5.center) -- ($(a5) + (disig) + (delta)$.center);

 \end{tikzpicture}
 \caption{An illustration of the chosen notation}
\end{figure}

\item We define 
 \begin{align*}
  \Delta_{k,l}^i &:= \sum_{\sigma > \tau} (e_i^\sigma - e_i) - \sum_{j=1}^d \alpha_j (e_i^{(j)} - e_i);\; i = 1,\dots, r-1,\\
  d_{k,l}^i &:= \sum_{\sigma > \tau} \left(\sum_{t = 1}^{r_{i,\sigma}} e_{i,t}^\sigma\right);\; i=1,\dots, r.
 \end{align*}
Summing up over all length differences at each vertex and edge and exchanging sums gives us the following equation:
\begin{align}
 \delta_{k,l}(\tau) &:= \sum_{\sigma > \tau} (\dist_{k,l}(v_\sigma) - \dist_{k,l}(p_0)) - \sum_{j=1}^d \alpha_j (\dist_{k,l}(p_j) - \dist_{k,l}(p_0))  \notag\\ 
 & = \sum_{i=1}^{r-1} (d_{k,l}^i + \Delta_{k,l}^i) + d_{k,l}^r .\label{eq:fibremorphism_eq_diff}
\end{align}
\end{itemize}
\end{notation}

\begin{remark}
 To prove that $\varphi_g$ is a pseudo-morphism, we need to show that $(\delta_{k,l})_{k<l} \in \Im(\phi_n)$, i.e.\ it is $0$ in $\mn$. The idea for the proof is the following: A cell $\rho$ that maps non-injectively onto some $\tau \in \curlyb$ (and thus carries edges of the fibres of the $p_i$) is a codimension one cell in $T$.  We will show that the vertices of the fibres in the surrounding maximal cells can be used to express the balancing condition of $\rho$ (lemma \ref{fibremorphism_lemma_normalvectors}). However, $\dim \rho = \dim \tau +1$, so we have an additional generator $w_i$ of $V_{\rho}$ (that generates the kernel of $g_{\mid \rho}$). We will then show that the quantities $\Delta_{k,l}^i$ and $d_{k,l}^i$ we defined above can be expressed in terms of the coordinates of the balancing equation in this element $w_i$ (lemma \ref{fibremorphism_lemma_diffidentity}). These expressions will then yield $\delta_{k,l}$ as an alternating sum where everything except the $w_i$-coefficients of the vertices at the leaves $k$ and $l$ cancels out.
\end{remark}

\begin{lemma}\label{fibremorphism_lemma_bijection}
 Let $\rho \in \curlyt$ be a cell such that $g(\rho) = \tau$ and $g_{\mid \rho}$ is not injective. Then there is a bijection
$$\Pi: \{\rho' > \rho\} \to \{\sigma > \tau\};\; \rho' \mapsto g(\rho')$$
\begin{proof}
 For surjectivity of $\Pi$, let $\sigma > \tau$. Choose elements $p \in \relint(\tau), q \in \relint(\sigma)$. By lemma \ref{fibremorphism_lemma_homeo}, $t_{q,p}^{-1}(g^*(p) \cap \rho)$ is a line segment. Let $\rho'$ be any cell containing an infinite subset of this. In particular, $g(\rho') = \sigma$. Then we can use the last statement of \ref{fibremorphism_lemma_homeo} to see that we must have $\rho' > \rho$.

For injectivity, assume that $g(\rho_1') = g(\rho_2') = \sigma > \tau$ for two distinct $\rho_i' > \rho$. Then $t_{q,p}(\betrag{g^*(q)} \cap \rho_i') = \betrag{g^*(p)} \cap \rho$ for $i=1,2$, which is a contradiction to the fourth statement of \ref{fibremorphism_lemma_homeo}.
\end{proof} 
\end{lemma}

\begin{lemma}\label{fibremorphism_lemma_normalvectors}
 Let $\rho \in \curlyt$ be a cell such that $q_i \in \rho$ and $\ker g_{\mid V_\rho} = \gnrt{w_i}$; i.e.\ $\rho$ contains (part of) the $i$-th edge. Then for any $\rho' > \rho$ we have $u_{\rho' / \rho} = q_{i,l}^\sigma - q_i$.

Similarly, if $\ker g_{\mid V_\rho} = \gnrt{w_{i-1}}$, then $u_{\rho'/\rho} = q_{i,k}^\sigma - q_i$.
\begin{proof}
We only consider the first case, since the second case is exactly analogous. By lemma \ref{fibremorphism_lemma_bijection}, there is a bijection
$$\Pi: \{\rho' > \rho\} \to \{\sigma > \tau\};\; \rho' \mapsto g(\rho').$$
Also, since $\bar{\lambda}_g$ is surjective, we have the following isomorphisms:
\begin{align*}
\Lambda_{\rho'} &\cong \Lambda_{g(\rho')} \times \gnrt{w_i} \textnormal{ for all } \rho'> \rho,\\
\Lambda_{\rho} &\cong \Lambda_{\tau} \times \gnrt{w_i} \\
\Longrightarrow\; \faktor{\Lambda_{\rho'}}{\Lambda_{\rho}} &\cong \faktor{\Lambda_{g(\rho')}}{\Lambda_\tau}.
\end{align*}
Since for any $\sigma > \tau$, $t_{v_\sigma,p_0}(q_{i,l}^\sigma) = q_i$ and the map preserves polyhedra, both vertices are contained in a common polyhedron which must be a face of $\rho' := \Pi^{-1}(\sigma)$. Hence $q_{i,l}^\sigma - q_i$ is a representative of $u_{\sigma' / \rho'} = (u_{\sigma/\tau},0) = (v_\sigma - p_0,0)$ using the isomorphism above.
\end{proof}
\end{lemma}

\begin{corollary}\label{fibremorphism_lemma_xisexist}
 For each $k \neq l \in [n]$, each $i = 1,\dots,r$, there exist $\xi_i(k,l), \chi_i(k,l) \in \R$ such that
\begin{align}
 \sum_{j=1}^d \alpha_j (q_i^{(j)} - q_i) &= \sum_{\sigma > \tau} (q_{i,l}^\sigma - q_i) + \xi_i(k,l) \cdot w_i, \label{eq:fibremorphism_eq_xi} \\
 \sum_{j=1}^d \alpha_j (q_i^{(j)} - q_i) &= \sum_{\sigma > \tau} (q_{i,k}^\sigma - q_i) + \chi_i(k,l) \cdot w_{i-1}.\label{eq:fibremorphism_eq_chi}
\end{align}
\begin{proof}
As in lemma \ref{fibremorphism_lemma_normalvectors}, choose $\rho \in \curlyt$ such that $q_i \in \rho$ and $\ker g_{\mid V_\rho} = \gnrt{w_i}$. Then we know that 
$$\sum_{\sigma > \tau} (q_{i,l}^\sigma - q_i) \in V_\rho .$$
Furthermore, by lemma \ref{fibremorphism_lemma_homeo}, $q_i,q_i^{(1)},\dots,q_i^{(d)}$ are all contained in a common face of $\rho$, hence 
$$\sum_{j=1}^d \alpha_j(q_i^{(j)} - q_i) \in V_\rho$$
as well. Since both sums map to the same element $\sum_{\sigma>\tau} (v_\sigma - p_0) = \sum_{j=1}^d \alpha_j (p_j - p_0)$ under $g$, they can only differ by an element from $\ker g_{\mid V_{\rho'}} = \gnrt{w_i}$, which implies the first equation. The second equation follows analogously.
\end{proof}
\end{corollary}

\begin{remark}\label{fibremorphism_remark_xi}
Since $w_0 = v_k$ is the same for all $l$, it is clear from the equations themselves that $\chi_1(k,l) = \chi_1(k)$ actually only depends on $k$. Similarly, $\xi_r$ only depends on $l$ and if we reverse the path direction, we find that 
$$\chi_1(k) = \chi_1(k,l) = - \xi_r(l,k).$$
\end{remark}

\begin{lemma}\label{fibremorphism_lemma_diffidentity}
 For each $k \neq l \in [n]$ we have
\begin{align*}
 \Delta_{k,l}^i &= \xi_i - \chi_{i+1} \textnormal{ for all } i = 1,\dots, r-1, \\
 d_{k,l}^i &= \chi_i - \xi_i \textnormal{ for all } i = 1,\dots,r.
\end{align*}
\begin{proof}
 If we subtract equation \eqref{eq:fibremorphism_eq_xi} from \eqref{eq:fibremorphism_eq_chi} for $i+1$, we obtain
\begin{align*}
 &\sum_{j=1}^d \alpha_j (\underbrace{(q_{i+1}^{(j)} - q_i^{(j)}) - (q_{i+1} - q_i)}_{= (e_i^{(j)} - e_i) \cdot w_i}) \\
= &\sum_{\sigma > \tau} (\underbrace{(q_{i+1,k}^\sigma - q_{i,l}^\sigma) - (q_{i+1} - q_i)}_{= (e_i^\sigma - e_i) \cdot w_i}) + (\chi_{i+1} - \xi_i) \cdot w_i .
\end{align*}
Factoring out $w_i$ we obtain
$$0 = \Delta_{k,l}^i - \xi_i + \chi_{i+1}.$$
For the second equation let $i \in \{1,\dots, r\}$ be arbitrary. Since $g^*(p_0)$ is a smooth curve, it is locally at $q_i$ isomorphic to $L_1^{\textnormal{val}(q_i)}$. Denote by $z_1,\dots, z_s$ the direction vectors of the outgoing edges, w.l.o.g.\ $z_1 = - w_{i-1}, z_s = w_i$. Now each edge $E$ in the preimage of $q_i$ under $t_{v_\sigma,p_0}$ induces a partition of the set $\{1,\dots ,s\} = I_E \mathaccent\cdot\cup I_E^c$ such that $x,y \in \{1,\dots,s\}$ are contained in the same set if and only if the path from $z_x$ to $z_y$ does not pass through $E$ (i.e.\ we separate the $z_i$ ``on one side of $E$'' from the others). It is easy to see that, due to the balancing condition of the curve, the direction vector of $E$ must be
$$w_E = \pm \sum_{x \in I_E} z_x = \mp \sum_{y \in I_E^c} z_y,$$
depending on the choice of orientation. One can, for example, see this by induction on the number of edges. Now assume $E$ lies on the path from $k$ to $l$; i.e.\ in $t_{v_\sigma,p_0}^{-1}(q_i)$ it lies on the path from $q_{i,k}^\sigma$ to $q_{i,l}^\sigma$. Choose $I_E$ such that $1 \notin I_E \ni s$, i.e.\ $w_E$ points towards $l$. 
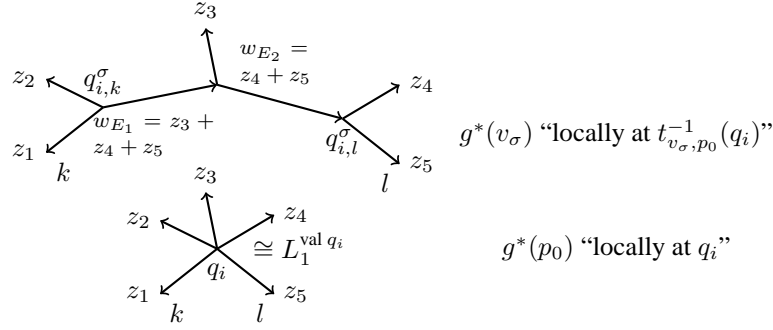
\begin{figure}[h]
 \centering
 \begin{tikzpicture}[scale = 0.75]
    \node (qi) at (0,0) {};
    \node (z1) at (-1,-0.8) {};
    \node (z2) at (-1,0.5) {};
    \node (z3) at (-0.2,1) {};
    \node (z4) at (1,0.6) {};
    \node (z5) at (1,-0.8) {};

    \node (qik) at (-2,2.5) {};
    \node (qim) at (0,2.9) {};
    \node (qil) at (2.2,2.3) {};

  \begin{scope}[thick]
    
    \foreach \x in {(z1.center),(z2.center),(z3.center),(z4.center),(z5.center)}
      \draw[->] (qi.center) -- \x;


    \draw[->] (qik.center) node[above]{$q_{i,k}^\sigma$} --
	  (qim.center) node[midway,below = 3pt,font = \footnotesize,
			    text width = 50pt]{$w_{E_1} = z_3 + z_4 + z_5$}; 
    \draw[->] (qim.center) -- 
	  (qil.center) node[below]{$q_{i,l}^\sigma$}
		       node[midway,above = 3pt,font = \footnotesize,
			  text width = 30pt]{$w_{E_2} = z_4 + z_5$};
    \draw[->] (qik.center) -- ($(qik) + (z1)$.center) node[left]{$z_1$} node[below right]{$k$};
    \draw[->] (qik.center) -- ($(qik) + (z2)$.center) node[left]{$z_2$};
    \draw[->] (qim.center) -- ($(qim) + (z3)$.center) node[above]{$z_3$};
    \draw[->] (qil.center) -- ($(qil) + (z4)$.center) node[right]{$z_4$};
    \draw[->] (qil.center) -- ($(qil) + (z5)$.center) node[right]{$z_5$} node[below left]{$l$};

  \end{scope}


  \draw (qi) node[below = 2pt] {$q_i$};
  \draw (qi) node[right = 10pt] {$\cong L_1^{\textnormal{val }q_i}$};
  \draw (7,0) node{$g^*(p_0)$ ``locally at $q_i$''};
  \draw (7,2) node{$g^*(v_\sigma)$ ``locally at $t_{v_\sigma,p_0}^{-1}(q_i)$''};
  \draw (z1) node[left]{$z_1$} node[below right]{$k$};
  \draw (z2) node[left]{$z_2$};
  \draw (z3) node[above]{$z_3$};
  \draw (z4) node[right]{$z_4$};
  \draw (z5) node[right]{$z_5$} node[below left]{$l$};

 \end{tikzpicture}
 \caption{The direction vector of an edge is determined by the $z_i$ lying ``behind'' it.}
\end{figure}
Denote by $E_1^\sigma,\dots,E_{r_{i,\sigma}}^\sigma$ the sequence of edges from $q_{i,k}^\sigma$ to $q_{i,l}^\sigma$. Subtracting equation \eqref{eq:fibremorphism_eq_xi} from \eqref{eq:fibremorphism_eq_chi} for the same $i$, we obtain
\begin{align*}
 0 =& \sum_{\sigma > \tau}(q_{i,l}^\sigma - q_{i,k}^\sigma) + \xi_i \cdot w_i - \chi_i \cdot w_{i-1} \\
=& \sum_{\sigma > \tau} \left(\sum_{t=1}^{r_{i,\sigma}} e_{i,t}^\sigma \cdot w_{E_t}\right) + \xi_i \cdot z_s + \chi_i \cdot z_1 \\
=& z_s \cdot \left(\sum_{\sigma > \tau} \left( \sum_{t=1}^r e_{i,t}^\sigma \right)\right) 
  + \underbrace{\sum_{\sigma > \tau} \left(\sum_{t=1}^r e_{i,t}^\sigma \left(\sum_{x \in I_{E_t}\setminus\{s\}} z_x\right) \right)}_{:= R, \textnormal{ contains neither } z_1 \textnormal{ nor } z_s}\\
  &+ \xi_i \cdot z_s + \chi_i \cdot z_1 \\
=& z_s \cdot (d_{k,l}^i + \xi_i) - \chi_i \left( \sum_{x \neq 1} z_x \right) + R .
\end{align*}
Since $z_1$ does no longer appear in this equation and $\{z_x, x \neq 1\}$ is linearly independent by smoothness, the coefficient of $z_s$ must be 0:
$$0 = d_{k,l}^i + \xi_i - \chi_i.$$
\end{proof}
\end{lemma}

\begin{proof}[Proof of proposition \ref{fibremorphism_prop_morphism}]
 By equation \eqref{eq:fibremorphism_eq_diff} and lemma \ref{fibremorphism_lemma_diffidentity} we have 
\begin{align*}
 \delta_{k,l}(\tau) &= \sum_{i=1}^{r-1} (d_{k,l}^i + \Delta_{k,l}^i) + d_{k,l}^r \\
		    &= \chi_1(k,l) - \xi_r(k,l) \\
		    &\stackrel{\ref{fibremorphism_remark_xi}}{=} \chi_1(k,l) + \chi_1(l,k) \\
		    &\stackrel{\ref{fibremorphism_remark_xi}}{=} \chi_1(k) + \chi_1(l).
\end{align*}
Hence $$(\delta_{k,l}(\tau))_{k<l} = \phi_n((\chi_1(r))_{r=1,\dots,n}).$$
\end{proof}

Before concluding this section we want to see that the second condition in our definition of prefamilies is really only necessary for cells on which $g$ is not injective (cf. remark \ref{families_remark_second}). Therefore, we first notice that the proofs of lemma \ref{fibremorphism_lemma_homeo} and lemma \ref{fibremorphism_lemma_bijection} do not use the second prefamily axiom.

\begin{lemma}
\label{lambdasurjective}
Let $B$ be a smooth variety and let $g:T\rightarrow B$ be a locally surjective morphism of tropical varieties all of whose fibres are smooth rational curves with $n$ unbounded edges. Let $\tau \in \curlyt$ be a cell on which $g$ is not injective. Then $\lambda_{g\mid\tau}:\Lambda_{\tau}\rightarrow\Lambda_{g(\tau)}$ is surjective. Moreover, all maximal cells in $\curlyt$ have trivial weight $1$.
\end{lemma}
\begin{proof}
%
%

We assume without loss of generality that $B$ is connected. As this implies that $B$ is irreducible (cf.\ \cite{francoisrau}*{lemma 2.4}) the bijection of lemma \ref{fibremorphism_lemma_bijection} implies that there is an integer $\lambda$ such that 
$\omega_{\curlyt}(\sigma)=\lambda\cdot\omega_{\curlyb}(g(\sigma))$ for all maximal cells in $\sigma\in\curlyt$. We thus need to show that $\lambda=1$ and that $g(v_{\sigma/\tau})=v_{g(\sigma)/g(\tau)}$ if $g$ is not injective on $\tau$, i.e.\ $g$ maps normal vectors to normal vectors. It is clear that $g(v_{\sigma/\tau})$ is a multiple of $v_{g(\sigma)/g(\tau)}$; as $B$ is a matroid variety, it follows that $g(v_{\sigma/\tau})=\lambda_{\tau}\cdot v_{g(\sigma)/g(\tau)}$ for some $\lambda_{\tau}\in\Z_{>0}$ which does not depend on $\sigma$. Let $\varphi_1\ldots,\varphi_{\dim(B)}$ be rational functions with $\varphi_1\cdots\varphi_{\dim(B)}\cdot B=\{0\}$ (cf.\ proof of lemma \ref{settheoreticfibre}).
Comparing the weight formulas for intersection products of $\omega_{\varphi_1\cdots\varphi_{\dim(B)}\cdot B}(\{0\})$ and $\omega_{g^*\varphi_1\cdots g^*\varphi_{\dim(B)}\cdot T}(\tau)$ for an edge $\tau\in\curlyt$, we see that $\lambda=1$ and $\lambda_{\beta}=1$ for all cells $\beta\geq\tau$.
\end{proof}

\section{Equivalence of families}\label{section_equivalence}

In the classical case, two families $T \stackrel{g}{\to} B, T' \stackrel{g'}{\to} B$ are equivalent if there is an isomorphism $\psi: T \to T'$ that commutes with the morphisms and markings. Such an isomorphism hence automatically induces isomorphisms between the fibres $g^*(p)$ and $g'^*(p)$ of a point $p$ in $B$.  

Recall that we call two families equivalent if their fibres over each point agree. We would like to show the existence of such an isomorphism $\psi:T\rightarrow T'$ for any two equivalent families. In fact, requiring $\psi$ to identify the fibres already uniquely fixes the map $\psi$, so for any two equivalent families of $n$-marked tropical curves we obtain a bijective map $T \to T'$ that commutes with $g,g'$ and the markings by identifying the fibres over each point $p$ (which are isomorphic by definition). We would like to see if this map is in fact a morphism. Again, we will only be able to show that it is a pseudo-morphism and since in general we can not assume $T$ to be smooth, we cannot give a stronger statement.

\begin{definition}
 Let $T \stackrel{g}{\to} B, T' \stackrel{g'}{\to} B$ be two equivalent families of $n$-marked tropical curves. Now for each point $p$ in $B$ there is a unique isomorphism of tropical curves $$\psi_p: g^*(p) \to g'^*(p)$$
(i.e.\ it identifies equally marked leaves and is linear of slope 1 on each edge). We define a map 
\begin{align*}
 \psi: T &\to T' \\
       t &\mapsto \psi_{g(t)}(t).
\end{align*}
\end{definition}

\begin{theorem}\label{equivalencethm}
 The map $\psi$ is a bijective pseudo-morphism whose inverse is also a pseudo-morphism. In particular, if $T$ or $T'$ is smooth, $\psi$ is an isomorphism.
\begin{proof}
 Since the construction of $\psi$ is symmetric, it is clear that the inverse of $\psi$ is a pseudo-morphism if $\psi$ itself is one. Also, by proposition \ref{fibremorphism_prop_smoothpseudo}, it is an isomorphism if any of $T$ or $T'$ is smooth.

First, we prove that $\psi$ is piecewise integer affine linear: Let $\tau\in\curlyt$ and choose $t\in \tau, t' \in \relint (\tau)$. Again, it suffices to show that $\psi$ is affine linear on the line segment $\conv \{t,t'\}$.

By lemma \ref{fibremorphism_lemma_homeo}, $t$ and $t'$ lie on edges of the corresponding fibres which have the same direction vector $w$. Select vertices $p,p'$ of these edges such that $t = p + \alpha \cdot w, t' = p' + \alpha' \cdot w$ for $\alpha, \alpha' \geq 0$.

Denote by $q := \psi(p), q' := \psi(p')$ and let $\xi$ be the direction vector of the corresponding edge in $T'$. Hence 
\begin{align*}
 \psi(t) = \psi(p + \alpha \cdot w) &= q + \alpha \cdot \xi \\
 \psi(t') = \psi(p' + \alpha' \cdot w) &= q' + \alpha' \cdot \xi
\end{align*}
and using the fact that any convex combination of $p$ and $p'$ must by \ref{fibremorphism_lemma_homeo} again be a vertex, it follows that
\begin{align*}
 \psi(t + \gamma(t' - t)) &= \psi((p + \gamma(p'-p)) + w \cdot (\alpha + \gamma(\alpha' - \alpha))) \\
			  &= (q + \gamma(q' - q)) + \xi \cdot (\alpha + \gamma (\alpha' - \alpha)) \\
			  &= \psi(t) + \gamma(\psi(t') - \psi(t))
\end{align*}
for any $\gamma \in [0,1]$. Hence $\psi$ is affine linear. Using the fact that it has slope 1 on each edge of a fibre and that $g' \circ \psi = g$, it is easy to see that it respects the lattice.

It remains to see that $\psi$ is a pseudo-morphism, so let $\tau$ be a codimension one cell of $T$. We distinguish two cases: 
\begin{itemize}
 \item $g_{\mid \tau}$ is injective: Then $g(\tau)$ is a maximal cell of $B$, so the adjacent maximal cells $\sigma > \tau$ are also mapped to $g(\tau)$. So if we take a point $p \in \relint(\tau)$, the normal vectors $v_{\sigma/\tau} - p$ correspond to normal vectors of the edges of the fibre $g^*(g(p))$ adjacent to $p$ (after proper refinement). Since the fibre is smooth, these add up to 0 and by definition of $\psi$, so do their images $\psi(v_{\sigma/\tau}) - \psi(p)$.
 \item $g_{\mid \tau}$ is not injective: Hence the fibre in $\tau$ over a generic point $p_0 \in g(\tau)$ is contained in the $m$-th edge on the path from some leaf $k$ to some leaf $l$ (it doesn't really matter, which one). Choose $p_0,\dots,p_d,v_\sigma$ in $g(\tau)$ and its adjacent cells $g(\sigma), \sigma > \tau$ as defined in \ref{fibremorphism_notation}. We now use the shorthand notation $q_0, q_j, q_\sigma$ for the $m$-th vertex point of the fibres of $p_0, p_j$ and $v_\sigma$. Now lemma \ref{fibremorphism_lemma_normalvectors} tells us that $q_\sigma - q_0$ is actually a normal vector of $\sigma$ with respect to $\tau$ and that its balancing equation reads
$$\;\;\;\;\;\sum_{\sigma > \tau} (q_\sigma - q_0) = \sum_{j=1}^d \alpha_j (q_j - q_0) - \xi_m^T(k,l)\cdot w_m.$$
Now the image of $q_0$ under $\psi$ is by definition the $m$-th vertex of the fibre $g'^*(p_0)$, so we also get
$$\;\;\;\;\;\;\;\;\;\;\;\;\;\;\;\sum_{\sigma > \tau} (\psi(q_\sigma) - \psi(q_0)) = \sum_{j=1}^d \alpha_j (\psi(q_j) - \psi(q_0)) - \xi_m^{T'}(k,l) \cdot \psi(w_m).$$
Hence, to prove that $\psi$ is a pseudo-morphism, it remains to show that $\xi_m^{T'}(k,l) = \xi_m^{T}(k,l)$.

By the proof of proposition \ref{fibremorphism_prop_morphism}, we know that 
$$\delta_{k,l}(\tau) = \phi_n((\chi_1^T(k))_{k=1,\dots,n}) = \phi_n((\chi_1^{T'}(k))_{k=1,\dots,n}). $$
Since the left side is independent on the choice of family by definition (it is defined only in terms of lengths of fibres) and $\Phi_n$ is injective, we must have $\chi_1^T(k) = \chi_1^{T'}(k)$ for any $k$. Using the fact that $d_{k,l}^i$ and $\Delta_{k,l}^i$ are also independent of the choice of family and applying lemma \ref{fibremorphism_lemma_diffidentity} inductively, we finally see that 
$$\chi_i^T(k,l) = \chi_i^{T'}(k,l) \textnormal{ and } \xi_i^T(k,l) = \xi_i^{T'}(k,l)$$
for any possible $i,k,l$.
\end{itemize}

\end{proof}

\end{theorem}

\begin{acknowledgement}
Georges Fran\c{c}ois is supported by the Fonds national de la Recherche (FNR), Luxembourg.

Simon Hampe is supported by the Deutsche Forschungsgemeinschaft grant GA 636 / 4-1
\end{acknowledgement}


\begin{bibdiv}
\begin{biblist}

\bib{ak}{article}{
	eprint = {\arxiv{math/0311370v2}},
	author = {Ardila, Frederico},
	author = {Klivans, Caroline J.},
        title = {The Bergman complex of a matroid and phylogenetic trees},
	journal = {J.\ Comb.\ Theory, Ser.\ B},
	volume = {96},
	pages = {38--49},
	year = {2006},
}

\bib{AR}{article}{
  author={Allermann, Lars},
  author={Rau, Johannes},
  title={First steps in tropical intersection theory},
  journal={Math.\ Z.},
  volume={264},
  number={3},
  pages={633--670},
  year={2010},
  eprint={\arxiv {0709.3705v3}},
}

\bib{francois}{article}{
  author={Fran\c{c}ois, Georges},
  title={Cocycles on tropical varieties via piecewise polynomials},
  eprint={\arxiv{1102.4783v2}},
}

\bib{francoisrau}{article}{
  author={Fran\c{c}ois, Georges},
  author={Rau, Johannes},
  title={The diagonal of tropical matroid varieties and cycle intersections},
  eprint={\arxiv{1012.3260v1}},
}


\bib{mapoly}{article}{
	eprint = {\arxiv{math/0411260}},
	author = {Feichtner, Eva Maria},
	author = {Sturmfels, Bernd},
        title = {Matroid polytopes, nested sets and Bergman fans},
	journal = {Port.\ Math.\ (N.S.)},
	volume = {62},
	pages = {437--468},
	year = {2005},
}

\bib{GKM}{article}{
  author={Gathmann, Andreas},
  author={Kerber, Michael},
  author={Markwig, Hannah},
  title={Tropical fans and the moduli spaces of tropical curves},
  journal={Compos.\ Math.},
  volume={145},
  number={1},
  pages={173--195},
  year={2009},
  eprint={\arxiv{0708.2268}},
}



\bib{kapranov}{article}{
  author={Kapranov, Mikhail},
  title={Chow quotients of Grassmannians I},
  journal={Adv.\ Soviet Math.},
  volume={16},
  pages={29--110},
  year={1993},
  eprint={\arxiv{alg-geom/9210002v1}},
}



\bib{psiclasses}{article}{
	eprint = {\arxiv{0709.3953v2}},
	author = {Kerber,Michael},
        author = {Markwig, Hannah},
	title = {Intersecting Psi-classes on tropical $\mathcal{M}_{0,n}$},
	journal = {Int.\ Math.\ Res.\ Notices},
	volume = {2009},
	number ={2},
	pages = {221--240},
	year = {2009},
}


\bib{KV}{book}{
  author={Kock, Joachim},
  author={Vainsencher, Israel},
  title={An Invitation to Quantum Cohomology},
  publisher={Progress in Mathematics 249},
  address={Birkh\"auser Boston},
  pages={159 p.},
  date={2007},
}

\bib{mikh}{article}{
	eprint = {\arxiv{math/0601041v2}},
	author = {Mikhalkin, Grigory},
        title = {Tropical geometry and its applications},
	journal = {Proceedings of the ICM, Madrid, Spain},
	pages = {827--852},
	year = {2006},
}



\bib{disshannes}{thesis}{
  author={Rau, Johannes},
  title={Tropical intersection theory and gravitational descendants},
  type={PhD thesis},
  organization={Technische Universit\"{a}t Kaiserslautern},
  date={2009},
  eprint={\href {http://kluedo.ub.uni-kl.de/volltexte/2009/2370/}{http://kluedo.ub.uni-kl.de/volltexte/2009/2370}},
}

\bib{shaw}{article}{
	eprint = {\arxiv{1010.3967v1}},
	author = {Shaw, Kristin},
        title = {A tropical intersection product in matroidal fans},
}


\bib{troplin}{article}{
	eprint = {\arxiv{math/0410455}},
	author = {Speyer, David},
        title = {Tropical linear spaces},
	journal = {SIAM J.\ Discrete Math.},
	volume = {22},
	pages = {1527--1558},
	year = {2008},
}


\bib{stusolve}{article}{
	author = {Sturmfels, Bernd},
        title = {Solving systems of polynomial equations},
	journal = {CBMS Regional Conferences Series in Mathematics, vol.\ 97, Published for the Conference Board of the Mathematical Sciences, Washington, DC},
	year = {2002},
}



\bib{SS04}{article}{
	eprint = {\arxiv{math/0304218v3}},
	author = {Speyer, David},
        author = {Sturmfels, Bernd},
	title = {The tropical Grassmannian},
	journal = {Adv.\ Geom. },
	pages = {389--411},
	volume = {4},
	year = {2004},
}



\bib{tevelev}{article}{
	eprint = {\arxiv{math/0412329v3}},
	author = {Tevelev, Jenia},
   title = {Compactifications of subvarieties of tori},
	journal = {Amer.\ J.\ Math.},
	pages = {1087--1104},
	volume = {129},
	year = {2007},
}

\end{biblist}
\end{bibdiv}
\end{document}